\documentclass[12pt]{amsart}
\usepackage{amsfonts,amsmath,amssymb,amscd}

\newcommand{\newsection}[1]{\setcounter{equation}{0} \section{#1}}
\newtheorem{theorem}{Theorem}[section]

\newtheorem{lemma}[theorem]{{\bf Lemma}}
\newtheorem{coro}[theorem]{{\bf Corollary}}

\newtheorem{remark}[theorem]{Remark}
\newtheorem{question}{Question}

\def \D{\mathbb{D}}
\newcommand{\clb}{\mathcal{B}}
\newcommand{\cle}{\mathcal{E}}
\newcommand{\clh}{\mathcal{H}}
\newcommand{\clw}{\mathcal{W}}
\newcommand{\clk}{\mathcal{K}}
\newcommand{\cls}{\mathcal{S}}
\newcommand{\clf}{\mathcal{F}}

\newcommand{\C}{\mbox{$\mathbb C$}}     

\newcommand{\ran}{\mbox{ran}}

\begin{document}

\title[Pairs of projections and isometries]
{Pairs of projections and commuting isometries}

\author[De]{Sandipan De}
\address{Sandipan De, School of Mathematics and Computer Science, Indian Institute of Technology Goa,
	Farmagudi, Ponda-403401, Goa, India.}
\email{sandipan@iitgoa.ac.in, 444sandipan@gmail.com}

\author[Sarkar]{Jaydeb Sarkar}
\address{Jaydeb Sarkar, Indian Statistical Institute,
Statistics and Mathematics Unit, 8th Mile, Mysore Road,
Bangalore, 560 059, India.}
\email{jaydeb@gmail.com, jay@isibang.ac.in}

\author[Shankar]{P Shankar}
\address{Shankar. P, Department of Mathematics, Cochin University of Science and Technology, Kochi 682022, Kerala, India.}
\email{shankarsupy@cusat.ac.in, shankarsupy@gmail.com}

\author[Sankar]{Sankar T.R.}
\address{Sankar T.R., Department of Mathematics, Indian Institute of Technology Bombay, Powai, Mumbai, 400076,
India}
\email{sankartr90@gmail.com}

\subjclass[2010]{47A13, 47A65, 47B47, 15A15, 47B35, 47B07}
\keywords{Shift operators, isometries, projections, weighted shifts, weighted shift matrices, Toeplitz operators, compact operators.}

\begin{abstract}
It is known that the non-zero part of compact defect operators of Berger-Coburn-Lebow pairs (BCL pairs in short) of isometries are diagonal operators of the form
\[
\begin{bmatrix}
I_1 & & & \\
& D & & \\
& & - I_2 & \\
& & & - D \\
\end{bmatrix},
\]
where $I_1$ and $I_2$ are the identity operators and $D$ is a positive contractive diagonal operator. We discuss the question of constructing an irreducible BCL pair from a diagonal operator of the above type. The answer to this question is sometimes in the affirmative and sometimes in the negative. This also answers a part of the question raised by He, Qin, and Yang. Our explicit constructions of BCL pairs yield concrete examples of pairs of commuting isometries.
\end{abstract}

\maketitle

\tableofcontents

\newsection{Introduction}

It is well known that the general theory of pairs of commuting isometries is complicated and the inadequacy of concrete representations of pairs of commuting isometries is a challenging obstacle to the comprehension of multivariable operator theory. In this paper, we focus on the Berger, Coburn, and Lebow pairs of commuting isometries \cite{BCL75} from the perspective of pairs of orthogonal projections (or in short projections) \cite{Davis, AP, DJS2016}, defect operators of commuting tuples of bounded linear operators \cite{Guo}, and a question of He, Qin, and Yang \cite{HQY2015}.

Isometries and projections are connected via the well known notion of defect operators \cite{Nagy book}. Let $V$ be a bounded linear operator on a Hilbert space $\clh$ ($V \in \clb(\clh)$ in short, and all Hilbert spaces are assumed to be separable and over $\C$). The \textit{defect operator} of $V$ is the linear operator $I - V V^*$. If $V$ is an isometry, then it is easy to see that
\[
I - V V^* = P_{\clw},
\]
the orthogonal projection onto the \textit{wandering subspace} $\clw := \ker V^*$. If, in addition, $V$ is a \textit{shift}, that is
\[
\bigcap_{n=0}^\infty V^n \clh = \{0\},
\]
then $V$ is unitarily equivalent to $M_z$ on $H^2_{\clw}(\D)$, where $M_z$ is the operator of multiplication by the coordinate function $z$, and $H^2_{\clw}(\D)$ denotes the $\clw$-valued Hardy space over the open unit disk $\mathbb{D} = \{z \in \C: |z| < 1\}$. Note that dimension of the wandering subspaces (or rank of defect operators) is the only unitary invariant of shift operators. Then the classical von Neumann Wold decomposition theorem \cite[Page 3, Theorem 1.1]{Nagy book} completely classifies the structure of isometries: An isometry is simply a shift or a unitary or a direct sum of a shift and a unitary. Since the structure of unitary operators is completely clear, the defect operator (or the wandering subspace) plays a crucial role in the classification of isometries.

Now we turn to pairs of commuting isometries. Unlike the case of isometries, the general structure and tractable invariants of pairs of commuting isometries are largely unknown (cf. \cite{Yang-S}). However, we still have a suitable notion of defect operator for tuples of isometries, which encodes a greater amount of information about operators \cite{Guo}. The \textit{defect operator} of a pair of commuting isometries $(V_1, V_2)$ is defined by
\[
C(V_1, V_2) = I - V_1 V_1^* - V_2 V_2^* + V_1 V_2 V_1^* V_2^*.
\]
In one hand, this notion has some resemblance to defect operators of single isometries, but on the other hand, the defect operator of a general pair of commuting isometries is fairly complex and difficult to analyze. However, the situation is somewhat favorable in the case of  Berger, Coburn, and Lebow pairs (BCL pairs in short): A commuting pair of isometries $(V_1, V_2)$ is said to be a \textit{BCL pair} if $V_1 V_2$ is a shift.

The main novelty in the definition of BCL pairs is the shift part, which brings analytic flavor to pairs of commuting isometries. Let $\cle$ be a Hilbert space, $U \in \clb(\cle)$ a unitary, and let $P \in \clb(\cle)$ be a projection. We call the ordered triple $(\mathcal{E}, U, P)$ a \textit{BCL triple}. Given a BCL triple $(\cle, U, P)$, we consider the pair of Toeplitz operators $(M_{\Phi_1}, M_{\Phi_2})$ on $H^2_{\cle}(\D)$ with analytic symbols
\begin{equation}\label{eqn: Phi analytic}
\Phi_1(z) = (P + z P^\perp)U^*, \text{ and } \Phi_2(z) = U(P^\perp + zP) \qquad (z \in \D),
\end{equation}
where $P^{\perp} :=I - P$. It is easy to see that
\[
M_{\Phi_1} M_{\Phi_2} = M_{\Phi_2} M_{\Phi_1} = M_z,
\]
and hence $(M_{\Phi_1}, M_{\Phi_2})$ on $H^2_{\cle}(\D)$ is a BCL pair. And, this is precisely the analytic model of BCL pairs \cite{BCL1978}: Up to joint unitary equivalence, BCL pairs are of the form $(M_{\Phi_1}, M_{\Phi_2})$ on $H^2_{\cle}(\D)$ for BCL triples $(\mathcal{E}, U, P)$.

\begin{remark}\label{remark-agreement}
In view of the above analytic model, throughout this paper, we will use BCL pair $(V_1, V_2)$ on $\clh$, BCL pair $(M_{\Phi_1}, M_{\Phi_2})$ on $H^2_{\cle}(\D)$ with $\Phi_1$ and $\Phi_2$ as in \eqref{eqn: Phi analytic}, and the associated BCL triple $(\cle, U, P)$ interchangeably.
\end{remark}

Returning to defect operators, for the BCL pair $(V_1, V_2)=(M_{\Phi_1}, M_{\Phi_2})$ defined as in \eqref{eqn: Phi analytic}, one finds that
\[
C(M_{\Phi_1}, M_{\Phi_2}) = \left[ \begin{array}{cc}
UPU^* - P & 0 \\
0 &0\\
\end{array}
\right],
\]
on $H^2_{\cle}(\D) = \cle \oplus z H^2_{\cle}(\D)$ \cite[Page 5]{HQY2015}, so that $z H^2_{\cle}(\D) \subseteq \ker C(M_{\Phi_1}, M_{\Phi_2})$. In particular, it suffices to study $C(M_{\Phi_1}, M_{\Phi_2})$ only on $\cle$. This and the above remark, then motivate us to define the \textit{defect operator} of the BCL triple $(\cle, U, P)$ as
\begin{equation}\label{eqn: def C}
C := C(M_{\Phi_1}, M_{\Phi_2})|_{\cle} = UPU^* - P.
\end{equation}
We shall reserve the symbol $C$ exclusively for defect operators associated to BCL triples. Clearly, $(UPU^*, P)$ is a pair of projections on $\cle$. Therefore, being a difference of a pair of projections, $C$ is a self-adjoint contraction (see \cite{AZ, AS, Andruchow1,ASS} for the general theory of pairs of projections). A natural question therefore arises: Does the difference of a pair of projections on some Hilbert space $\cle$ determine a BCL pair $(M_{\Phi_1}, M_{\Phi_2})$ on $H^2_{\cle}(\D)$? Evidently, in this generality, this problem is less accessible and a resolution seems to be despairing.

At this point, we return to the above setting and observe, in addition, that if $C$ is compact, then $C|_{(\ker C)^\perp}$ is unitarily equivalent to a special diagonal operator: A compact diagonal operator $T$ on a Hilbert space is said to be a \textit{distinguished diagonal operator} if
\begin{equation}\label{eq-intro defect}
T= \begin{bmatrix}
I_1 & & & \\
& D & & \\
& & - I_2 & \\
& & & - D \\
\end{bmatrix},
\end{equation}
where $I_1$ and $I_2$ are the identity operators and $D$ is a positive contractive diagonal operator. It is important to note that, up to unitary equivalence, a distinguished diagonal operator always can be represented as a difference of two projections (see Theorem \ref{struc}). Then, in view of present terminology, Theorem 4.3 of \cite{HQY2015}, which is also the entry point of this paper, states:

\begin{theorem}[He, Qin and Yang]\label{structure}
Let $(M_{\Phi_1}, M_{\Phi_2})$ be a BCL pair. If $C(M_{\Phi_1}, M_{\Phi_2})$ is compact, then its non-zero part is unitarily equivalent to a distinguished diagonal operator.
\end{theorem}

The goal of this paper, largely, is to suggest the missing link between distinguished diagonal operators and BCL pairs. More specifically, given a distinguished diagonal operator $T \in \clb(\cle)$, we are interested in constructing BCL pairs $(M_{\Phi_1}, M_{\Phi_2})$ on $H^2_{\cle}(\D)$ such that the non-zero part of $C(M_{\Phi_1}, M_{\Phi_2})$ is equal to $T$. However, in order to avoid trivial situation (cf. \cite[Theorem 6.7]{HQY2015}), we need to impose the irreducibility condition on the pairs: A pair of bounded linear operators on a Hilbert space is said to be \textit{irreducible} if the only closed subspaces that reduce both the operators are the trivial ones.

We will see in Corollary \ref{cor-irred}, a BCL pair $(M_{\Phi_1}, M_{\Phi_2})$ on $H^2_{\cle}(\D)$ is irreducible if and only if $(\cle, U, P)$ is irreducible (that is, the pair $(U, P)$ on $\cle$ is irreducible). Therefore, irreducibility is compatible with BCL pairs and BCL triples. The following is the central question of this paper.

\begin{question}\label{Q1}
Let $T \in \clb(\cle)$ be a distinguished diagonal operator. Does there exist an irreducible BCL pair $(M_{\Phi_1}, M_{\Phi_2})$ on $H^2_{\cle}(\D)$ such that $C(M_{\Phi_1}, M_{\Phi_2})|_{\cle} = T$? Or, equivalently, does there exist an irreducible BCL triple $(\cle, U, P)$ such that $U P U^* - P = T$
\end{question}

It is worth noting that the injectivity of $T$ and the condition that $C(M_{\Phi_1}, M_{\Phi_2})|_{\cle} = T$ forces that
\[
(\ker C(M_{\Phi_1}, M_{\Phi_2}))^\perp = \cle.
\]

The above question also relates to an unresolved question raised by He, Qin, and Yang in \cite[page 18]{HQY2015}, which asks: Given a distinguished diagonal operator $T$, does there exist an irreducible BCL pair on some Hilbert space such that the non-zero part of the corresponding defect operator is unitarily equivalent to $T$? From this perspective, Question \ref{Q1} seeks for the irreducible BCL pair $(M_{\Phi_1}, M_{\Phi_2})$ with an additional property that $(\ker C(M_{\Phi_1}, M_{\Phi_2}))^\perp = \cle$. Evidently, an affirmative answer to Question \ref{Q1} would imply an affirmative answer to He, Qin, and Yang question.

We prove that the answer to Question \ref{Q1} is sometimes in the affirmative and sometimes in the negative. In order to be more precise, we proceed to elaborate on the spectral decomposition of defect operators. For $X \in \clb(\clh)$, we denote $\sigma(X)$ the spectrum of $X$, and for $\mu \in \mathbb{C}$, we denote
\[
E_{\mu}(X) = \ker(X - \mu I_{\clh}).
\]
Note again that the defect operator $C (= C(M_{\Phi_1}, M_{\Phi_2})|_{\cle})$ is a self-adjoint contraction. In addition, if $C$ is compact, then for each non-zero $\lambda \in \sigma(C) \cap (-1, 1)$, $-\lambda$ is also in $\sigma(C)$, and (see \cite[Lemma 4.2]{HQY2015})
\begin{equation}\label{eqn - plus minus lambda}
k_{\lambda} := \mbox{dim} E_{\lambda}(C) = \mbox{dim} E_{-\lambda}(C).
\end{equation}
Consequently, one can decompose $(\ker C)^{\perp}$ as
\[
(\ker C)^{\perp} = E_1(C) \oplus (\oplus_{\lambda} E_{\lambda}(C)) \oplus E_{-1}(C)
\oplus (\oplus_{\lambda} E_{-\lambda}(C)),
\]
where $\lambda$ runs over the set $\sigma(C) \cap (0, 1)$. Then
\[
C|_{(\ker C)^{\perp}} =
\begin{bmatrix}
I_{E_1} & & & \\
& \oplus_{\lambda} \lambda I_{E_{\lambda}} & & \\
& & -I_{E_{-1}} & \\
& & & \oplus_{\lambda} (-\lambda) I_{E_{-\lambda}} \\
\end{bmatrix},
\]
and hence $C|_{(\ker C)^{\perp}}$ is unitarily equivalent to a distinguished diagonal operator. More specifically
\begin{equation*}
[C|_{(\ker C)^{\perp}}] \cong
\begin{bmatrix}
I_{l_1} & & & \\
& D & & \\
& & -I_{l_1^{\prime}} & \\
& & & -D \\
\end{bmatrix},
\end{equation*}
where $l_1 = \mbox{dim} E_1(C)$, $l_1^{\prime} = \mbox{dim} E_{-1}(C)$, $D = \bigoplus_{\lambda} \lambda I_{k_{\lambda}}$, and for $m \in \mathbb{N}$, $I_m$ denotes the $m \times m$ identity matrix.

We are now ready to explain the main contribution of this paper. In Theorem \ref{2.4}, we prove a noteworthy property for finite-dimensional Hilbert spaces: Let $\cle$ be a finite-dimensional Hilbert space and let $(\cle, U, P)$ be a BCL triple. Then
\[
\dim E_1(C) = \dim E_{-1}(C).
\]
Corollary \ref{equ} then states that if $T \in \clb(\cle)$ is a distinguished diagonal operator and
\[
\mbox{dim} E_1(T) \neq \mbox{dim} E_{-1}(T),
\]
then it is not possible to find any (reducible or irreducible) BCL pair on $H^2_{\cle}(\mathbb{D})$ such that the non-zero part of the defect operator is unitarily equivalent to $T$. Therefore, the answer to Question \ref{Q1} is negative in this case. These results are the main content of Section \ref{section-basic results}.

In Section \ref{section-finite dim}, we initiate our investigation in search of an affirmative answer to Question \ref{Q1}. Here we deal with distinguished diagonal operators on finite-dimensional Hilbert spaces with at least two distinct positive eigenvalues. In the following section, Section \ref{section-finite dim 2}, we settle the remaining case, that is, distinguished diagonal operators with only one positive eigenvalue. The results of Section \ref{section-finite dim} and Section \ref{section-finite dim 2} summarize as follows (see Theorem \ref{main}): Let $\cle$ be a finite-dimensional Hilbert space, $T \in \clb(\mathcal{E})$ be a distinguished diagonal operator, and suppose
\[
\mbox{dim} E_1(T) = \mbox{dim} E_{-1}(T).
\]
If $T$ has either at least two distinct positive eigenvalues or, only one positive eigenvalue lying in $(0, 1)$, then there exists an irreducible BCL pair $(M_{\Phi_1}, M_{\Phi_2})$ on $H^2_{\cle}(\mathbb{D})$ such that $C(M_{\Phi_1}, M_{\Phi_2})|_{\cle} = T$. On the other extreme, suppose $1$ is the only positive eigenvalue of $T$. If
\[
\mbox{dim} E_1(T)= 1,
\]
then there exists an irreducible BCL pair $(M_{\Phi_1}, M_{\Phi_2})$ on $H^2_{\cle}(\mathbb{D})$ such that $C(M_{\Phi_1}, M_{\Phi_2})|_{\cle} = T$, and if
\[
\mbox{dim} E_1(T)> 1,
\]
then such an irreducible BCL pair does not exist.

\noindent Therefore, the results of Sections \ref{section-basic results}, \ref{section-finite dim}, and \ref{section-finite dim 2} completely settle Question \ref{Q1} in the case when $\cle$ is a finite-dimensional Hilbert space (also see the paragraph preceding Theorem \ref{main}).

Finally, in Section \ref{section-infinite dim} we deal with the case when $\mathcal{E}$ is infinite-dimensional. We prove that Question \ref{Q1} has an affirmative answer for the case when (see Theorem \ref{inf})
\[
\mbox{dim} E_1(T) = \mbox{dim} E_{-1}(T),
\]
as well as when (see Theorem \ref{diff1})
\[
\mbox{dim} E_1(T) = \mbox{dim} E_{-1}(T) \pm 1.
\]
Therefore, Question \ref{Q1} remains open for the remaining cases: $\cle$ is an infinite-dimensional Hilbert space, and $T \in \clb(\cle)$ a distinguished diagonal operator for which
\[
| \mbox{dim} E_1(T) - \mbox{dim} E_{-1}(T) | \geq 2.
\]

\newsection{Preparatory results}\label{sec: Preparatory}

In this section, we introduce some standard notations and prove some basic results that will be frequently used in the main body of the paper.

Recall, in view of \eqref{eqn: Phi analytic}, up to unitary equivalence, a BCL pair $(V_1, V_2)$ admits the analytic  representation $(V_1, V_2) = (M_{\Phi_1}, M_{\Phi_2})$, where
\begin{equation}\label{def}
\begin{aligned}
M_{\Phi_1} = (P + M_z P^\perp)U^*, \text{ and } M_{\Phi_2} = U(P^\perp + M_z P),
\end{aligned}
\end{equation}
for some BCL triple $(\cle, U, P)$ (also see Remark \ref{remark-agreement}). In particular, $V_1 V_2 = M_{\Phi_1} M_{\Phi_2} = M_z$.

The following lemma characterizes joint reducing subspaces of BCL pairs via joint reducing subspaces of BCL triples and vice versa.

\begin{lemma}\label{irr}
Let $(\cle, U, P)$ be a BCL triple, and let $\mathcal{S} \subseteq {H}^2_{\mathcal{E}}(\D)$ be a closed subspace. Then $\mathcal{S}$ reduces $(M_{\Phi_1}, M_{\Phi_2})$ if and only if there exists a closed reducing subspace $\tilde{\mathcal{E}} \subseteq \mathcal{E}$ for $( U, P)$ such that $\mathcal{S} = H^2_{\tilde{\mathcal{E}}}(\D)$.
\end{lemma}
\begin{proof}
If $\mathcal{S}$ reduces $(M_{\Phi_1}, M_{\Phi_2})$, then $\cls$ reduces $M_z$ (as $M_{\Phi_1} M_{\Phi_2} = M_z$), and hence there exists a closed subspace $\tilde{\mathcal{E}} \subseteq \mathcal{E}$ such that $\cls = {H}^2_{\tilde{\mathcal{E}}}(\D)$ \cite[Page 4, Corollary C]{Rovnyak}. It remains to show that $\tilde{\mathcal{E}}$  reduces $(U, P)$. Let $\eta \in \tilde{\mathcal{E}}$. By \eqref{def}, we know that
\[
M_{\Phi_1} \eta = PU^* \eta + (P^{\bot}U^* \eta) z,
\]
is a one-degree polynomial in $H^2_{\tilde{\cle}}(\D)$. So we conclude that
\[
U^* \eta = PU^* \eta + P^{\perp}U^* \eta \in \tilde{\mathcal{E}} \qquad (\eta \in \tilde{\cle}).
\]
Therefore, $PU^*$, $P^{\perp}U^*$ and (hence) $U^*$ leave $\tilde{\mathcal{E}}$ invariant. Similarly, using
\[
M_{\Phi_2} \eta = UP^{\bot}\eta + (UP \eta)z \in H^2_{\tilde{\cle}}(\D),
\]
we conclude that $UP^{\bot}$ and $UP$ leave $\tilde{\cle}$ invariant. Then $U (= UP^\perp + UP)$ leaves $\tilde{\mathcal{E}}$ invariant, and hence $\tilde{\cle}$ reduces $U$. Finally, $PU^* \tilde{\cle} \subseteq \tilde{\cle}$, $UP \tilde{\cle} \subseteq \tilde{\cle}$, and
\[
P  = (PU^*)(UP),
\]
imply that $\tilde{\cle}$ reduces $P$. The converse simply follows from the representations in \eqref{def}.
\end{proof}

Recall that a BCL triple $(\cle, U,P)$ is said to be \textit{irreducible} if the pair $(U,P)$ on $\cle$ is irreducible. The following is now straightforward:

\begin{coro}\label{cor-irred}
$(M_{\Phi_1}, M_{\Phi_2})$ on $H^2_{\cle}(\D)$ is irreducible if and only if $(\cle, U,P)$ is irreducible.
\end{coro}

For convenience in what follows, we introduce another layer of notation: The set of all ordered orthonormal bases of a Hilbert space $\mathcal{H}$ will be denoted by $B_{\clh}$. For instance, if $\{e_j : j \in \mathbb{Z}\}$ is an orthonormal basis of $l^2(\mathbb{Z})$, then we simply write
\[
\{e_j : j \in \mathbb{Z}\} \in B_{l^2(\mathbb{Z})}.
\]

Weighted shifts and weighted shift matrices will be core objects in our analysis. Let $\{\lambda_j\}_{j\in \mathbb{Z}}$ be a sequence of non-zero scalars. An operator $S \in \clb(\clh)$ is called a \textit{weighted shift} \cite{Shields} with weight sequence $\{\lambda_j\}_{j\in \mathbb{Z}}$ if there exists $\{e_j : j \in \mathbb{Z}\} \in B_{\mathcal{H}}$ such that
\begin{align*}
S e_i = \lambda_i e_{i+1} \quad \quad (i \in \mathbb{Z}).
\end{align*}
For $x \in \clh$, we say that $x$ is a star-cyclic vector for $S$ if
\[
\overline{\text{span}}\{S^m x, S^{*m} x: m \geq 0 \} = \clh.
\]
The finite-dimensional counterpart of weighted shifts is the so called weighted shift matrix \cite{Yu}: If $\{e_j: 1 \leq j \leq n\} \in B_{\mathcal{H}}$, then
\[
S e_j =
\begin{cases} \lambda_i e_{j+1} & \mbox{if}~ 1 \leq j < n
\\
\lambda_n e_1 & \mbox{if}~ j=n, \end{cases}
\]
is called a \textit{weighted shift matrix} with weight $\{\lambda_j\}_{j=1}^n$. For notational simplicity we denote the matrix of $S$ with respect to the ordered basis $\{e_j\}_{j=1}^n \in B_{\clh}$ as $[\lambda_n; \lambda_1, \ldots, \lambda_{n-1}]$, that is
\begin{equation}\label{eqn:weighted matrix 1}
[S] = [\lambda_n; \lambda_1, \ldots, \lambda_{n-1}] = \begin{bmatrix}
0 & 0 & \cdots & 0 & \lambda_n
\\
\lambda_1 & 0 & \cdots & 0 & 0
\\
0 & \lambda_2  & \cdots & 0 & 0
\\
\vdots & \vdots & \ddots & \vdots & \vdots
\\
0 & 0 & \cdots & \lambda_{n-1} & 0
\end{bmatrix}.
\end{equation}
Moreover, if $\lambda_j = \lambda$ for all $j=1, \ldots, n-1$, then we simply write the above as $[\lambda_n; J_{n-1}(\lambda)] $, that is
\[
[\lambda_n; J_{n-1}(\lambda)] = \begin{bmatrix}
0 & 0 & \cdots & 0 & \lambda_n
\\
\lambda & 0 & \cdots & 0 & 0
\\
0 & \lambda  & \cdots & 0 & 0
\\
\vdots & \vdots & \ddots & \vdots & \vdots
\\
0 & 0 & \cdots & \lambda & 0
\end{bmatrix}.
\]
Finally, we introduce the following general notation. Consider the weighted shift matrix $[\lambda_n; \lambda_1, \ldots, \lambda_{n-1}]$ corresponding to the weights $\{\lambda_i\}_{i=1}^n$. Suppose
\[
\{\lambda_{i_1}, \lambda_{i_2}, \ldots, \lambda_{i_m}: 1 = i_1 < i_2 < \cdots < i_m = {n-1}\},
\]
be the set of distinct elements of $\{\lambda_1, \ldots, \lambda_{n-1}\}$ and suppose $\lambda_{i_t}$ occurs $k_t$-times, $t = 1, \ldots, m$. Then we write the corresponding weighted shift matrix as
\begin{equation}\label{eqn:weighted matrix 2}
[\lambda_n; \lambda_1, \ldots, \lambda_{n-1}] = [\lambda_n; J_{k_1}(\lambda_{i_1}), \ldots, J_{k_m}(\lambda_{i_{m}})].
\end{equation}
The above elaboration of notation of weighted shift matrices will turn out to be helpful in the computation part of this paper. Also, the following elementary property of weighted shift matrices will be used repeatedly.

\begin{lemma}\label{lemma: weighted matrix - diagonal}
Let $S$ be an $n \times n$ weighted shift matrix, and let $1 \leq j \leq n$. Then
\[
S^j = \begin{bmatrix}
0 & D_{j} \\
D_{n-j} & 0
\end{bmatrix},
\]
where $D_j$ and $D_{n-j}$ are respectively $j \times j$ and $(n-j) \times (n-j)$ diagonal matrices with non-zero entries on the diagonals.
\end{lemma}
\begin{proof}
The proof follows by a straightforward induction
\end{proof}

We also need the star-cyclicity and the cyclicity property of weighted shifts and weighted shift matrices, respectively, in what follows. The result may be well known but we shall give the proof for the sake of completeness.

\begin{lemma}\label{shift}
(i) Let $S$ be a weighted shift matrix corresponding to $\{e_i\}_{i=1}^n \in B_{\clh}$. Then $e_j$ is a cyclic vector of $S$ for all $j \in \{1, \ldots, n\}$. (ii) If $S$ is a weighted shift corresponding to $\{e_i\}_{i \in \mathbb{Z}} \in B_{\clh}$, then $e_j$ is a star-cyclic vector of $S$ for all $j \in \mathbb{Z}$.
\end{lemma}
\begin{proof}
Suppose $S$ is a weighted shift matrix on $\clh$ corresponding to $\{e_i\}_{i=1}^n \in B_{\clh}$. Fix $j \in \{1, \ldots, n\}$. By repeated application of \eqref{eqn:weighted matrix 1}, one can easily see that
\[
S^p e_j = \begin{cases}
\alpha_{j+p} e_{j+p} & \text{ if } p=1, \ldots, n-j \\
\alpha_{p+j-n} e_{p+j-n} & \text{ if } p = n-j+1, \ldots, n,
\end{cases}
\]
where $\alpha_i$'s are non-zero scalar. In particular,
\[
\text{span} \{S^p e_j: p=1, \ldots, n\} = \text{span} \{e_t: t=1, \ldots, n\} = \clh,
\]
and hence $e_j$ is a cyclic vector of the weighted shift matrix $S$.

\noindent Now suppose $S$ is a weighted shift corresponding to $\{e_i\}_{i \in \mathbb{Z}} \in B_{\clh}$, and $j \in \mathbb{Z}$ is fixed. In this case, it follows that
\[
S^p e_j = \alpha_{m+j} e_{j+p} \qquad (p \geq 0),
\]
and
\[
S^{*q} e_j = \alpha_{j-q} e_{j-q} \qquad (q > 0),
\]
where $\alpha_i$'s are non-zero scalars. Clearly, as in the weighted shift matrix case, this yields the desired result.
\end{proof}

Let $\cls$ be a closed subspace of $\clh$. To avoid confusion of notation we denote by $Q_{\cls}$ the orthogonal projection of $\clh$ onto $\cls$. The following  elementary fact will be frequently used in the irreducibility part of BCL pairs.

\begin{lemma}\label{lemma: proj reducing}
Let $\cls$ be a closed subspace of $\clh$, and let $P \in \clb(\clh)$ be a projection. Then $\cls$ reduces $P$ if and only if $\cls$ reduces $P^\perp$.
\end{lemma}
\begin{proof}
Since $P^\perp$ is also a projection, it is enough to prove the lemma in one direction. Suppose that $\cls$ reduces $P$. By assumption, $Q_{\cls} P = P Q_{\cls}$, and hence
\[
Q_{\cls} P^\perp = Q_{\cls} (I - P) = Q_{\cls} - P Q_{\cls} = P^\perp Q_{\cls}.
\]
This proves that $\cls$ reduces $P^\perp$.
\end{proof}

Let $T \in \clb(\clh)$ be a compact self-adjoint operator. We know that $\sigma(T) = \{\lambda_i : i \in \Lambda\}$, for some countable set $\Lambda$, and the spectral decomposition of $T$ as
\[
\clh = \bigoplus_{i \in \Lambda} E_{\lambda_i}(T),
\]
We have the following simple and well known property:

\begin{lemma}\label{lemma: diagonal}
Let $\mathcal{S}$ be a closed $T$-invariant subspace of $\mathcal{H}$. If $\oplus_{i \in \Lambda} x_i  \in \cls$, then $x_i \in \mathcal{S}$ for all $i \in \Lambda$.
\end{lemma}
\begin{proof}
Since $T$ is self-adjoint, we know that $Q_{\mathcal{S}} T = T Q_{\cls}$. Fix $i \in \Lambda$. Then
\[
T Q_{\mathcal{S}} y_i = Q_{\cls} T y_i = \lambda_i Q_{\cls} y_i,
\]
for all $y_i \in E_{\lambda_i}(T)$ implies that $E_{\lambda_i}(T)$ is invariant under $Q_{\mathcal{S}}$, and hence $E_{\lambda_i}(T)$ reduces $Q_{\cls}$. This says that
\[
Q_{\cls} Q_{E_{\lambda_i}(T)} = Q_{E_{\lambda_i}(T)} Q_{\cls} \qquad (i \in \Lambda).
\]
In particular, if $x = \oplus_{i \in \Lambda} x_i \in \mathcal{S}$, then
\[
x_i = Q_{E_{\lambda_i}(T)} x = Q_{E_{\lambda_i}(T)} Q_{\cls} x = Q_{\cls} Q_{E_{\lambda_i}(T)} x = Q_{\cls} x_i,
\]
and hence, $x_i \in \cls$ for all $i \in \Lambda$.
\end{proof}

\newsection{Eigenspaces of different dimensions}\label{section-basic results}

In this section, we prove that the answer to Question \ref{Q1} is negative in general. Our construction is a byproduct of certain dimension inequality. More specifically, in Corollary \ref{equ}, we prove that Question \ref{Q1} is in negative for all distinguished diagonal operators $T$ acting on finite-dimensional Hilbert spaces for which
\[
\mbox{dim}\, E_1(T) \neq \mbox{dim}\, E_{-1}(T).
\]

We start with the structure of operators that can be expressed as the difference of two projections \cite{DJS2016}. Let $A \in \clb(\mathcal{H})$ be a self-adjoint contraction. Then $\ker A$, $\ker(A-I)$ and $\ker (A+I)$ reduces $A$ \cite{Hal}. Hence there exists a closed subspace $\clh_0 \subseteq \clh$ such that $\clh$ admit the direct sum decomposition
\[
\clh = \ker A \oplus \ker(A-I) \oplus \ker (A+I) \oplus \mathcal{H}_0.
\]
Let us now assume that $\mathcal{H}_0 = \mathcal{K} \oplus \mathcal{K}$ for some Hilbert space $\mathcal{K}$, and suppose, with respect to
\[
\mathcal{H} = \ker A \oplus \ker(A-I) \oplus \ker(A+I) \oplus \mathcal{K} \oplus \mathcal{K},
\]
$A$ admits the block-diagonal operator matrix representation
\[
A =
\begin{bmatrix}
0 & & & \\
& I & & \\
& & -I & \\
& & & D & \\
& & & & -D\\
\end{bmatrix},
\]
for some positive contraction $D \in \clb(\mathcal{K})$. The following comes from \cite[Theorem 3.2]{DJS2016}:

\begin{theorem}\label{struc}
With notations as above, the diagonal operator $A$ is a difference of two projections. Moreover, if $A = P - Q$ for some projections $P$ and $Q$, then there exists a projection $R \in \clb(\ker A)$ such that
\begin{align*}
P = R \oplus I \oplus 0 \oplus P_U \quad \mbox{and} \quad Q = R \oplus 0 \oplus I \oplus Q_U,
\end{align*}
where \begin{equation*}
P_U = \frac{1}{2} \left[ \begin{array}{cc}
I + D & U(I-D^2)^{\frac{1}{2}} \\
U^* (I-D^2)^{\frac{1}{2}} & I - D\\
\end{array}
\right]
\mbox{~and~}
Q_U = \frac{1}{2} \left[ \begin{array}{cc}
I - D & U(I-D^2)^{\frac{1}{2}} \\
U^* (I-D^2)^{\frac{1}{2}} & I + D\\
\end{array}
\right],
\end{equation*}
are projections in $\clb(\mathcal{K} \oplus \mathcal{K})$,	and $U \in \clb(\mathcal{K})$ is a unitary commuting with $D$.
\end{theorem}

Therefore, given a diagonal operator $A$ as above, Theorem \ref{struc} parameterizes pairs of projections in terms of the positive contraction $D$ and a unitary $U \in \{D\}'$, whose differences are $A$. In particular, if $D = \lambda I_{\mathcal{K}}$ for some $\lambda \in [0, 1]$, then
\begin{equation}\label{eqn: P U}
P_U = \left[ \begin{array}{cc}
\frac{1+\lambda}{2}I_{\clk} & \frac{\sqrt{1-\lambda^2}}{2}U
\\
\frac{\sqrt{1-\lambda^2}}{2} U^* & \frac{1-\lambda}{2}I_{\clk}
\end{array}
\right].
\end{equation}
The following result, in particular, presents an orthonormal basis of the range space of $P_U$.

\begin{lemma}\label{range}
Let $\mathcal{H}$ and $\mathcal{K}$ be Hilbert spaces, $U : \mathcal{H} \rightarrow \mathcal{K}$ a unitary operator, and let $\lambda \in [0, 1]$. Define the projection $P : \mathcal{H} \oplus \mathcal{K} \rightarrow \mathcal{H} \oplus \mathcal{K}$ by
\begin{equation*}
P = \left[ \begin{array}{cc}
\frac{1+\lambda}{2}I_{\mathcal{H}} & \frac{\sqrt{1-\lambda^2}}{2} U^*
\\
\frac{\sqrt{1-\lambda^2}}{2} U & \frac{1-\lambda}{2}I_{\mathcal{K}}
\\
\end{array}
\right].
\end{equation*}
If $\{ e_i : i \in \Lambda \} \in B_{\mathcal{H}}$, then $\Big\{ \sqrt{\frac{1+\lambda}{2}} e_i \oplus \sqrt{\frac{1-\lambda}{2}} Ue_i : i \in \Lambda \Big \} \in B_{\ran P}$.
\end{lemma}
\begin{proof}
For each $x \in \mathcal{H}$, a simple calculation shows that
\begin{equation}\label{eqn: Px}
P(x \oplus 0) = \frac{1+\lambda}{2}x \oplus \frac{\sqrt{1-\lambda^2}}{2} Ux =
P\Big(0 \oplus \sqrt{\frac{1+\lambda}{1-\lambda}}Ux\Big).
\end{equation}
By duality
\begin{equation}\label{eqn: Py}
P(0 \oplus y) = \frac{\sqrt{1-\lambda^2}}{2} U^*y \oplus \frac{1-\lambda}{2}y = P\Big(\sqrt{\frac{1-\lambda}{1+\lambda}}U^*y \oplus 0\Big),
\end{equation}
for all $y \in \mathcal{K}$. The above equalities imply that
\begin{equation}\label{eq-ran P}
\mbox{ran} P = \{P(x \oplus 0) : x \in \mathcal{H}\} = \{P(0 \oplus y) : y \in \mathcal{K}\}.
\end{equation}
On the other hand, since $\{e_i : i \in \Lambda \} \in B_{\mathcal{H}}$, by \eqref{eqn: Px} we have
\[
\|P(e_i \oplus 0)\| = \sqrt{\frac{1+\lambda}{2}} \qquad (i \in \Lambda).
\]
Then \eqref{eq-ran P} and the first equality of \eqref{eqn: Px} readily implies that $\Big \{\sqrt{\frac{1+\lambda}{2}} e_i \oplus \sqrt{\frac{1-\lambda}{2}} Ue_i : i \in \Lambda\Big \} \in B_{\ran P}$, which completes the proof of the lemma.
\end{proof}

Similarly, \eqref{eq-ran P} and the first equality of \eqref{eqn: Py} implies that $\Big \{\sqrt{\frac{1-\lambda}{2}} U^* e_i \oplus  \sqrt{\frac{1+\lambda}{2}} e_i : i \in \Lambda\Big \} \in B_{\ran P}$. We also note that the index set $\Lambda$ is at most countable.

Recall that $C = UPU^* - P$ is the defect operator of the BCL triple $(\cle, U, P)$ (see \eqref{eqn: def C}). Then, using the notation $P^{\perp} = I - P$, we find
\begin{equation}\label{eq-defect other one}
C = P^{\perp} - U P^{\perp} U^*.
\end{equation}
The following appears to be a distinctive property of defect operators on finite-dimensional spaces.

\begin{theorem}\label{2.4}
Let $(\cle, U, P)$ be a BCL triple. If $\cle$ is finite-dimensional, then
\[
\dim E_1(C) = \dim E_{-1}(C).
\]
\end{theorem}
\begin{proof}
Suppose $\sigma(C) \cap (0,1) = \{\lambda_i : 1 \leq i \leq m \}$ (possibly an empty set). By \cite[Lemma 4.2]{HQY2015} (or, see \eqref{eqn - plus minus lambda}), it follows that $- \lambda_i \in \sigma(C)$ and
\[
\dim E_{\lambda_i}(C) = \dim E_{-\lambda_i}(C) \qquad (i = 1, 2, \cdots, m).
\]
Then, for each $i = 1, \ldots, m$, choose a unitary $U_i : E_{-\lambda_i}(C) \rightarrow E_{\lambda_i}(C)$. We set
\[
\clk_+ = \bigoplus_{i = 1}^m E_{\lambda_i}(C), \text{ and } \clk_- = \bigoplus_{i = 1}^m E_{-\lambda_i}(C),
\]
and $U^{\prime}: = \bigoplus_{i=1}^m U_i$. Therefore, $U': \clk_- \rightarrow \clk_+$ is a unitary. Also, set
\[
\tilde{\mathcal{E}}: = E_0(C) \oplus E_1(C) \oplus E_{-1}(C) \oplus \clk_+ \oplus \clk_+.
\]
Clearly
\[
{\mathcal{E}} = E_0(C) \oplus E_1(C) \oplus \clk_+ \oplus E_{-1}(C) \oplus \clk_-,
\]
and hence
\[
W =
\begin{bmatrix}
I_{E_0(C)} & 0 & 0 & 0 & 0\\
0 & I_{E_1(C)} & 0 & 0 & 0
\\
0 & 0 & 0 & I_{E_{-1}(C)} & 0
\\
0 & 0 & I_{\clk_+} & 0 & 0
\\
0 & 0 & 0 & 0 & U^{\prime}
\end{bmatrix},
\]
defines a unitary $W: \mathcal{E} \rightarrow \tilde{\cle}$. If we define $\tilde{C} := W C W^*$, then a simple calculation shows that $\tilde{C}$ is a diagonal operator
\[
\tilde{C} =
\begin{bmatrix}
0_{E_0(C)} & & & & \\
& I_{E_1(C)} & & & \\
& & - I_{E_{-1}(C)} & & \\
& & & X & \\
& & & & -X
\end{bmatrix},
\]
where
\[
X := \bigoplus_{i = 1}^m \lambda_i I_{E_{\lambda_i}(C)}.
\]
Moreover, $P_1 := W P^\perp W^*$ and $P_2 :=  W (U P^\perp U^*) W^*$ define projections in $\clb(\tilde{\cle})$, and
\[
\tilde{C} = W C W^* = W (P^\perp - U P^\perp U^*) W^* = P_1 - P_2.
\]
By Theorem \ref{struc}, there exist a projection $R \in \clb(E_0(C))$ and a unitary $Y \in \clb(\clk_+)$ commuting with $X$ such that
\begin{align*}
P_1 = R \oplus I_{E_1(C)} \oplus 0_{E_{-1}(C)} \oplus P_Y, \text{ and } \ P_2 = R \oplus 0_{E_1(C)} \oplus I_{E_{-1}(C)} \oplus Q_Y,
\end{align*}
where the projections $P_Y$ and $Q_Y$ on $\clk_+ \oplus \clk_+$ are given by
\[
P_Y = \frac{1}{2} \left[ \begin{array}{cc}
I + X & Y(I-X^2)^{\frac{1}{2}} \\
Y^* (I-X^2)^{\frac{1}{2}} & I - X\\
\end{array}
\right],
\mbox{~and~}
Q_Y = \frac{1}{2} \left[ \begin{array}{cc}
I - X & Y(I-X^2)^{\frac{1}{2}} \\
Y^* (I-X^2)^{\frac{1}{2}} & I + X\\
\end{array}
\right].
\]
From the definition of $X$ above, we have
\[
\frac{1}{2}(I \pm X) = \bigoplus_{i = 1}^m \Big(\frac{1 \pm \lambda_i}{2}I_{E_{\lambda_i}(C)}\Big),
\]
and
\[
\frac{1}{2}(I - X^2)^{\frac{1}{2}} = \bigoplus_{i = 1}^m \Big(\frac{\sqrt{1-\lambda_i^2}}{2} I_{E_{\lambda_i}(C)}\Big).
\]
Going back to the proof of Lemma \ref{range}, a closer inspection reveals that equalities similar to \eqref{eqn: Px} and \eqref{eqn: Py} also hold in the present setting. Indeed, if $x = \oplus_{i=1}^m x_i\in \clk_+$, then
\[
\begin{split}
P_Y(x \oplus 0) & = \Big(\oplus_{i=1}^m \frac{1+\lambda_i}{2}x_i\Big) \oplus \Big(\oplus_{i=1}^m \frac{\sqrt{1-\lambda_i^2}}{2} Y^* x_i \Big)
\\
& = P_Y \Big(0 \oplus \Big(\oplus_{i=1}^m \sqrt{\frac{1+\lambda_i}{1-\lambda_i}} Y^*x_i\Big)\Big),
\end{split}
\]
and by duality
\[
\begin{split}
Q_Y(0 \oplus x) & = \Big(\oplus_{i=1}^m \frac{\sqrt{1-\lambda_i^2}}{2} Y x_i \Big) \oplus \Big(\oplus_{i=1}^m \frac{1+\lambda_i}{2}x_i\Big)
\\
& = Q_Y \Big(\Big(\oplus_{i=1}^m \sqrt{\frac{1+\lambda_i}{1-\lambda_i}} Yx_i\Big) \oplus 0\Big).
\end{split}
\]
So we find (as similar to \eqref{eq-ran P})
\[
\ran P_Y = \{P_Y(x \oplus 0) : x \in \clk_+\}, \text{ and } \ran\, Q_Y = \{ Q_Y(0 \oplus x) : x \in \clk_+\}.
\]
Moreover, the vectors on the right-hand sides of $P_Y(x \oplus 0)$ and $Q_Y(0 \oplus x)$ in the above pair of equalities readily imply that $\tau: \ran P_Y \rightarrow \ran Q_Y$ defined by
\[
\tau (P_Y(x \oplus 0)) = Q_Y(0 \oplus x) \qquad (x \in \clk_+),
\]
is a linear isomorphism. In particular, $\mbox{rank} P_Y = \mbox{rank} Q_Y$. Also note that
\[
\mbox{rank} P_1 = \mbox{rank} (P^{\perp}) = \mbox{rank} (UP^{\perp}U^*) = \mbox{rank} P_2.
\]
Finally, since
\[
\mbox{rank}P_1 = \mbox{rank} R + \mbox{dim} E_1(C) + \mbox{rank} P_Y,
\]
and
\[
\mbox{rank} P_2 = \mbox{rank} R + \mbox{dim} E_{-1}(C) + \mbox{rank} Q_Y,
\]
we must have that $\mbox{dim}E_1(C) = \dim E_{-1}(C)$. This completes the proof of the theorem.
\end{proof}

We are now ready to prove that Question \ref{Q1} is negative in general.

\begin{coro}\label{equ}
Let $\cle$ be a finite-dimensional Hilbert space and let $T \in \clb(\cle)$ be a distinguished diagonal operator. If
\[
\mbox{dim}\, E_1(T) \neq \mbox{dim}\, E_{-1}(T),
\]
then it is not possible to find a BCL pair $(M_{\Phi_1}, M_{\Phi_2})$ on $H^2_{\cle}(\mathbb{D})$ such that the non-zero part of $C(M_{\Phi_1}, M_{\Phi_2})$ is unitarily equivalent to $T$.
\end{coro}
\begin{proof}
Assume by contradiction that $(M_{\Phi_1}, M_{\Phi_2})$ is a BCL pair on $H^2_{\cle}(\mathbb{D})$ such
\[
C(M_{\Phi_1}, M_{\Phi_2})|_{(\ker C(M_{\Phi_1}, M_{\Phi_2}))^{\perp}} \cong T.
\]
Then $\text{dim} \cle = \text{dim} (\ker C(M_{\Phi_1}, M_{\Phi_2}))^{\perp}$. Since $\cle$ is finite-dimensional, and
\[
(\ker C(M_{\Phi_1}, M_{\Phi_2}))^{\perp} \subseteq \cle,
\]
it follows that
\[
\cle = (\ker C(M_{\Phi_1}, M_{\Phi_2}))^{\perp},
\]
and hence
\[
T \cong C(M_{\Phi_1}, M_{\Phi_2})|_{(\ker C(M_{\Phi_1}, M_{\Phi_2}))^{\perp}} = C \in \clb(\cle).
\]
An appeal to Theorem \ref{2.4} then says that
\begin{align*}
\dim E_1(T) = \dim E_1(C) = \dim E_{-1}(C) = \dim E_{-1}(T),
\end{align*} 	
which is absurd.
\end{proof}

Then, for a hope of an affirmative answer to Question \ref{Q1} for finite-dimensional spaces, we must assume that $\mbox{dim}\, E_1(T) = \mbox{dim}\, E_{-1}(T)$. We settle this issue in the following two sections.

\newsection{Diagonals with more than one positive eigenvalue}\label{section-finite dim}

Let $\mathcal{E}$ be a finite-dimensional Hilbert space, and let $T \in \clb(\cle)$ be a distinguished diagonal operator. In view of Corollary \ref{equ}, it is natural to ask: Does Question \ref{Q1} has an affirmative answer whenever
\[
\mbox{dim} E_1(T) = \mbox{dim} E_{-1}(T)?
\]
As we will see in this and the following sections, the answer to this question is still case-based. Note that in view of Corollary \ref{cor-irred}, for an affirmative answer to Question \ref{Q1} it is enough to construct an irreducible BCL triple $(\cle, U, P)$ such that $UPU^* - P = T$.

From now on in this section, we assume that $\cle$ is a finite-dimensional Hilbert space and $T \in \clb(\cle)$ is a distinguished diagonal operator such that
\[
\mbox{dim} E_1(T) = \mbox{dim} E_{-1}(T).
\]
We begin with the scrutinization of the effect of $T$ on the geometry of $\cle$. It is clear from the definition of distinguished diagonal operators that
\[
\sigma(T) = \{\pm \lambda_j : j \in \Lambda \},
\]
where $\Lambda = \{1, 2, \cdots, n\}$ for some $n \in \mathbb{N}$ and $0 < \lambda_j \leq 1$ for all $j=1, \ldots, n$ (note that $T$ is injective). Then, by the spectral theorem, we have
\[
\mathcal{E} = \oplus_{i \in \Lambda}\Big(E_{\lambda_i}(T) \oplus E_{-\lambda_i}(T)\Big).
\]
Moreover, for each $i \in \Lambda$, \eqref{eqn - plus minus lambda} implies that
\[
k_i := \mbox{dim} E_{\lambda_i}(T) = \mbox{dim} E_{-\lambda_i}(T) < \infty.
\]
In this case, note that the dimension of $\cle$ is an even number. Now we collect a number of basic facts about $T$ and the finite-dimensional Hilbert space $\cle$.

\begin{lemma}\label{lemma: main fd}
For each $i \in \Lambda$, fix a unitary $U_i: E_{\lambda_i}(T) \rightarrow E_{-\lambda_i}(T)$, and a basis $\{e^i_t: t = 1, \ldots, k_i \} \in B_{E_{\lambda_i}(T)}$. Then:
\begin{enumerate}
\item $\cup_{i \in \Lambda} \{\sqrt{1-\lambda_i^2} f^i_t \oplus (- \lambda_i) \tilde{f}^i_t, \; \lambda_i f^i_t \oplus \sqrt{1-\lambda_i^2} \tilde{f}^i_t : t =1, \ldots, k_i \} \in B_{\mathcal{E}}$, where
\[
f^i_t := \sqrt{\frac{1+\lambda_i}{2}} e^i_t \oplus \sqrt{\frac{1-\lambda_i}{2}} U_ie^i_t \text{ and }
\tilde{f}^i_t := \sqrt{\frac{1-\lambda_i}{2}} e^i_t \oplus \Big(- \sqrt{\frac{1+\lambda_i}{2}}\Big) U_ie^i_t,
\]
for all $t = 1, \ldots, k_i$ and $i \in \Lambda$.

\item $Q_i = \left[ \begin{array}{cc}
\frac{1+\lambda_i}{2}I_{E_{\lambda_i}(T)} & \frac{\sqrt{1-\lambda_i^2}}{2} U_i^*
\\
\frac{\sqrt{1-\lambda_i^2}}{2} U_i & \frac{1-\lambda_i}{2}I_{E_{-\lambda_i}(T)}
\end{array}
\right]$ is a projection on $E_{\lambda_i}(T) \oplus E_{-\lambda_i}(T)$, $i \in \Lambda$, and
\[
P = \bigoplus_{i \in \Lambda} (I_{E_{\lambda_i}(T) \oplus E_{-\lambda_i}(T)} - Q_i),
\]
is a projection on $\cle$.

\item $\cup_{i \in \Lambda} \{\tilde{f}^i_t: t=1, \ldots, k_i \} \in B_{\ran P}$, and $\cup_{i \in \Lambda} \{f^i_t: t=1, \ldots, k_i \} \in B_{\ran P^{\bot}}$.

\item For each $t =1, \ldots, k_i$ and $i \in \Lambda$, we have
\[
T f^i_t = \lambda_i^2 f^i_t + \lambda_i \sqrt{1 - \lambda_i^2} \tilde{f}^i_t, \text{ and } T \tilde{f}^i_t = \lambda_i \sqrt{1 - \lambda_i^2} f^i_t
- \lambda_i^2 \tilde{f}^i_t.
\]
\end{enumerate}
\end{lemma}
\begin{proof}
By Theorem \ref{struc}, or more specifically, by \eqref{eqn: P U}, we conclude that
\[
Q_i := \left[ \begin{array}{cc}
\frac{1+\lambda_i}{2}I_{E_{\lambda_i}(T)} & \frac{\sqrt{1-\lambda_i^2}}{2} U_i^*
\\
\frac{\sqrt{1-\lambda_i^2}}{2} U_i & \frac{1-\lambda_i}{2}I_{E_{-\lambda_i}(T)}
\end{array}
\right] \in \clb(E_{\lambda_i}(T) \oplus E_{-\lambda_i}(T)),
\]
defines a projection on $E_{\lambda_i}(T) \oplus E_{-\lambda_i}(T)$. Let $P := Q^{\bot} \in \clb(\cle)$, where $Q :=\oplus_{i \in \Lambda} Q_i \in \clb(\mathcal{E})$. Then
\[
P = \bigoplus_{i \in \Lambda} \Big(I_{E_{\lambda_i}(T) \oplus E_{-\lambda_i}(T)} - Q_i\Big),
\]
is a projection in $\clb(\cle)$. Next, we fix an $i \in \Lambda$. Since $\{e^i_t: t = 1, \ldots, k_i \} \in B_{E_{\lambda_i}(T)}$, it follows that $\{U_i e^i_t: t = 1, \ldots, k_i \} \in B_{E_{-\lambda_i}(T)}$, and hence, by Lemma \ref{range}, $\{f^i_t: t = 1, \ldots, k_i\} \in  B_{\ran\, Q_i}$, where
\[
f^i_t := \sqrt{\frac{1+\lambda_i}{2}} e^i_t \oplus \sqrt{\frac{1-\lambda_i}{2}} U_ie^i_t \qquad (t = 1, \ldots, k_i).
\]
Similarly, Lemma \ref{range} applied to $I - Q_i$ yields that $\{\tilde{f}^i_t: t = 1, \ldots, k_i\} \in B_{\ran\, Q_i^{\bot}}$, where
\[
\tilde{f}^i_t := \sqrt{\frac{1-\lambda_i}{2}} e^i_t \oplus \Big(- \sqrt{\frac{1+\lambda_i}{2}}\Big) U_ie^i_t \qquad (t = 1, \ldots, k_i).
\]
Therefore
\[
\{f^i_t, \tilde{f}^i_t: t = 1, \ldots, k_i\} \in B_{E_{\lambda_i}(T) \oplus E_{-\lambda_i}(T)},
\]
and hence, from the definition of $P$, it follows that
\[
\cup_{i \in \Lambda} \{\tilde{f}^i_t: t=1, \ldots, k_i \} \in B_{\ran P}, \text{ and } \cup_{i \in \Lambda} \{f^i_t: t=1, \ldots, k_i \} \in B_{\ran P^{\bot}}.
\]
Now by changing $\lambda_i$ to $-\lambda_i$, we have
\[
\Big \{ \sqrt{\frac{1 - \lambda_i}{2}} \ e^i_t \oplus \sqrt{\frac{1 + \lambda_i}{2}} \ U_i e^i_t, \; \sqrt{\frac{1 + \lambda_i}{2}} \ e^i_t \oplus \Big(- \sqrt{\frac{1 - \lambda_i}{2}}\Big) \ U_i e^i_t: t =1, \ldots, k_i \Big \},
\]
is in $B_{E_{\lambda_i}(T) \oplus E_{-\lambda_i}(T)}$ for all $i \in \Lambda$. Since
\[
\sqrt{\frac{1 - \lambda_i}{2}} e^i_t \oplus \sqrt{\frac{1 + \lambda_i}{2}} U_i e^i_t = \sqrt{1-\lambda_i^2} f^i_t - \lambda_i \tilde{f}^i_t,
\]
and
\[
\sqrt{\frac{1 + \lambda_i}{2}} e^i_t \oplus \Big(- \sqrt{\frac{1 - \lambda_i}{2}} U_i e^i_t\Big) = \lambda_i f^i_t + \sqrt{1-\lambda_i^2} \tilde{f}^i_t,
\]
for all $i$ and $t$, we conclude that
\[
\cup_{i \in \Lambda} \{\sqrt{1-\lambda_i^2} f^i_t \oplus (- \lambda_i) \tilde{f}^i_t, \; \lambda_i f^i_t \oplus \sqrt{1-\lambda_i^2} \tilde{f}^i_t : t =1, \ldots, k_i \} \in B_{\mathcal{E}}.
\]
Finally, since
\[
T e_t^i = \lambda_i e^i_t, \text{ and } T (U e^i_t) = - \lambda_i (U e^i_t),
\]
an easy combination of the basis vectors $f^i_t$ and $\tilde{f}^i_t$ defined as above implies that
\[
T f^i_t = \lambda_i^2 f^i_t + \lambda_i \sqrt{1 - \lambda_i^2} \tilde{f}^i_t, \text{ and } T \tilde{f}^i_t = \lambda_i \sqrt{1 - \lambda_i^2} f^i_t
- \lambda_i^2 \tilde{f}^i_t,
\]
for all $t =1, \ldots, k_i$ and $i \in \Lambda$. This completes the proof of the lemma.
\end{proof}

Before we prove that Question \ref{Q1} has an affirmative answer whenever $\sigma(T)$ has at least two distinct positive numbers, let us introduce the following notation: Given a scalar $\lambda$ and a natural number $m$, we denote by $[\lambda]_m$ the $m \times m$ constant diagonal matrix with diagonal entry $\lambda$. That is
\begin{equation}\label{eq-lambda n}
[\lambda]_m =
\begin{bmatrix}
\lambda  & &
\\
&  \ddots & \\
& & \lambda
\end{bmatrix}_{m \times m}.
\end{equation}

\begin{theorem}\label{thm: main 1}
Let $\mathcal{E}$ be a finite-dimensional Hilbert space, $T \in \clb(\cle)$ be a distinguished diagonal operator and suppose $\mbox{dim} E_1(T) = \mbox{dim} E_{-1}(T)$. If $T$ has at least two distinct positive eigenvalues, then Question \ref{Q1} has an affirmative answer.
\end{theorem}

\begin{proof}
We need to construct an irreducible BCL triple $(\cle, U, P)$ such that $T = P^\perp - U P^\perp U^*$ (see \eqref{eq-defect other one}). We continue to work in the setting of Lemma \ref{lemma: main fd}, and consider the projection $P$ constructed in the same lemma. Recall that
\[
\sigma(T) = \{\pm \lambda_i: i \in \Lambda\},
\]
where $\Lambda = \{1, \ldots, n\}$. We now construct the required unitary $U \in \clb(\mathcal{E})$. In view of $\cle = \ran P^{\perp} \oplus \ran P$ and Lemma \ref{lemma: main fd}(3), define $U$ on $\ran P^{\perp}$ by
\[
U f^i_t = \sqrt{1 - \lambda_i^2}f^i_t \oplus (- \lambda_i) \tilde{f}^i_t,
\]
for all $t = 1, \ldots, k_i$ and $i=1, \ldots, n$, and define $U$ on $\ran P$ by
\[
U \tilde{f}^i_t =
\begin{cases} \lambda_i f^i_{t+1} \oplus (\sqrt{1 - \lambda_i^2}) \tilde{f}^i_{t+1} & \mbox{if}~ 1 \leq t < k_i \, \mbox{and} \, 1 \leq i \leq n
\\
\lambda_{i+1} f^{i+1}_1 \oplus (\sqrt{1-\lambda_{i+1}^2}) \tilde{f}^{i+1}_1 & \mbox{if}~ t = k_i \, \mbox{and} \, 1 \leq i < n
\\
\lambda_1 f^1_1 \oplus (\sqrt{1 - \lambda_1^2}) \tilde{f}^1_1 & \mbox{if}~ t= k_n \, \mbox{and} \, i= n.
\end{cases}
\]
Then $U$ maps an orthonormal basis of $\cle$ to an orthonormal basis of $\cle$, and hence $U \in \clb(\cle)$ is a unitary. Suppose
 \begin{equation*}
U = \left[ \begin{array}{cc}
U_{11} & U_{12} \\
U_{21} & U_{22}
\end{array}
\right],
\end{equation*}
on $\mathcal{E} = \ran P^{\bot} \oplus \ran P$. We now compute matrix representations of the entries $\{U_{ij}\}_{i,j=1}^2$ with respect to the ordered orthonormal bases $\cup_{i = 1}^n \{f^i_t : 1 \leq t \leq k_i \} \in B_{\ran P^{\perp}}$ and $\cup_{i = 1}^n \{\tilde{f}^i_t : 1 \leq t \leq k_i \} \in B_{\ran P}$ (see part (3) of Lemma \ref{lemma: main fd}). Observe that $U_{11} = Q_{\ran P^\perp} U|_{\ran P^\perp}$ and $U_{21} = Q_{\ran P} U|_{\ran P^\perp}$, and hence we obtain the matrix representations of $U_{11}: \ran P^{\perp} \rightarrow \ran P^{\perp}$ and $U_{21}: \ran P^{\perp} \rightarrow \ran P$ with respect to the above orthonormal bases of $ \ran P^{\perp}$ and $ \ran P$ as
\[
[U_{11}] = \mbox{diag}\Big( \Big[\sqrt{1-\lambda_1^2}\Big]_{k_1}, \ldots, \Big[\sqrt{1-\lambda_n^2}\Big]_{k_n} \Big),
\]
and (see \eqref{eq-lambda n} for the notation)
\[
[U_{21}] = \mbox{diag}\Big([-\lambda_1]_{k_1}, \ldots, [-\lambda_n]_{k_n} \Big).
\]
Moreover, since $U_{12}  = Q_{\ran P^\perp} U|_{\ran P}$ and $U_{22} = Q_{\ran P} U|_{\ran P}$, it follows that $U_{12}: \ran P \rightarrow \ran P^{\perp}$ and $U_{22}: \ran P \rightarrow \ran P$ admit weighted shift matrix representations as (see \eqref{eqn:weighted matrix 1} and \eqref{eqn:weighted matrix 2} for the notation used below)
\[
[U_{12}] = \Big[\lambda_1; J_{k_1-1}(\lambda_1), J_{k_2}(\lambda_2), \ldots, J_{k_n}(\lambda_n) \Big],
\]
and
\[
[U_{22}] = \Big[\sqrt{1 - \lambda_1^2}; J_{k_1-1}\Big(\sqrt{1 - \lambda_1^2}\Big), J_{k_2}\Big(\sqrt{1 - \lambda_2^2}\Big), \ldots, J_{k_n}\Big(\sqrt{1 - \lambda_n^2}\Big) \Big].
\]
Next, we verify that $P^{\perp} - U P^{\perp} U^* = T$. Let
\begin{equation*}
T = \left[ \begin{array}{cc}
T_{11} & T_{12} \\
T_{21} & T_{22}
\end{array}
\right],
\end{equation*}
on $\mathcal{E} = \ran P^{\bot} \oplus \ran P$. Since
\[
P^\perp - U P^\perp U^* = \begin{bmatrix}
I_{\ran P^{\perp}} - U_{11} U_{11}^* & - U_{11} U_{21}^*
\\
-U_{21} U_{11}^* & -U_{21} U_{21}^*
\end{bmatrix},
\]
the verification of $P^{\perp} - U P^{\perp} U^* = T$ amounts to verify the following equality
\begin{equation}\label{eq}
\left[ \begin{array}{cc}
T_{11} & T_{12} \\
T_{21} & T_{22}
\end{array}
\right] = \begin{bmatrix}
I_{\ran P^{\perp}} - U_{11} U_{11}^* & - U_{11} U_{21}^*
\\
-U_{21} U_{11}^* & -U_{21} U_{21}^*
\end{bmatrix}.
\end{equation}
Note that $U_{11}$ and $U_{21}$ are the only matrices that appear in the entries of the right side matrix. Since $U_{11}$ and $U_{21}$ are diagonal operators, it is easy to conclude that
\[
[I_{\ran P^{\perp}} - U_{11} U_{11}^*] = \mbox{diag}\Big([\lambda_1^2]_{k_1}, \ldots, [\lambda_n^2]_{k_n} \Big),
\]
and
\[
- [U_{11} U_{21}^*] = -[U_{21} U_{11}^*] = \mbox{diag}\Big( \Big[\lambda_1 \sqrt{1-\lambda_1^2}\Big]_{k_1}, \ldots, \Big[\lambda_n \sqrt{1-\lambda_n^2}\Big]_{k_n} \Big).
\]
By part (3) of Lemma \ref{lemma: main fd}, we find that the action of $T$ on the bases $\cup_{i = 1}^n \{f^i_t : 1 \leq t \leq k_i \} \in B_{\ran P^{\perp}}$ and $\cup_{i = 1}^n \{\tilde{f}^i_t : 1 \leq t \leq k_i \} \in B_{\ran P}$ forces $T_{ij}$, $i,j=1,2$, to be diagonal operators and yields that the corresponding entries on either side of \eqref{eq} are the same. This completes the proof of the fact that $P^{\perp} - U P^{\perp} U^* = T$.

Now all we need to verify is that $(U,P)$ is irreducible. Let $\mathcal{S}$ be a non-zero subspace of $\mathcal{E}$, and suppose that $\cls$ reduces $(U,P)$. We write
\[
\mathcal{S} = P^\perp \cls \oplus P \cls.
\]
Since $\cls$ is non-zero, either $P^\perp \cls \neq \{0\}$ or $P \cls \neq \{0\}$. Suppose $P^\perp \cls \neq \{0\}$. For the irreducibility of $(U, P)$, it is enough to prove that $\mathcal{S} = \mathcal{E}$ (the other case $P \cls \neq \{0\}$ will follow similarly). To this end, assume for a moment that $\ran P^{\perp} = P^\perp \cls$. Since $\lambda_i \neq 0$ for all $i \in \Lambda$, it follows that the diagonal operator $U_{21}: \ran P^{\perp} \rightarrow \ran P$ (see the definition of $U_{21}$ above) is a linear isomorphism. On the other hand, since $\cls$ reduces $(U, P)$, by Lemma \ref{lemma: proj reducing} we have
\begin{equation}\label{eq: U21 P perp S}
U_{21} (P^\perp \cls) = Q_{\ran P} U|_{\ran P^\perp} (P^\perp \cls) = Q_{\ran P} U (P^\perp \cls) \subseteq Q_{\ran P} \cls = P \cls.
\end{equation}
Combined together, these facts yield
\[
P \cls \supseteq U_{21}(P^{\perp} \cls) = U_{21}(\ran P^{\perp}) = \ran P,
\]
and hence $\ran P = P \cls$. Then
\[
\mathcal{S} = P^\perp \cls \oplus P \cls = \ran P^\perp \oplus \ran P = \cle,
\]
will prove that $(U, P)$ is irreducible. Therefore, for the irreducibility of $(U, P)$, it suffices to prove that
\[
\ran P^{\perp} = P^\perp \cls.
\]
To this end, observe first that by the matrix representations of $U_{12}$ and $U_{21}$, it follows that
\begin{equation}\label{formula U12U21}
U_{12} U_{21} {f}^i_t =
\begin{cases} - \lambda_i^2 f^i_{t+1} & \mbox{if}~ 1 \leq i \leq n \, \mbox{and} \, 1 \leq t < k_i
\\
- \lambda_{i+1}^2 f^{i+1}_1 & \mbox{if}~ 1 \leq i < n \, \mbox{and} \,t = k_1
\\
- \lambda_1 \lambda_n f^1_1 & \mbox{if}~ i=n \, \mbox{and} \, t= k_n,
\end{cases}
\end{equation}
and hence
\begin{equation}\label{eq: U 1221 matrix}
U_{12} U_{21} = \Big[- \lambda_1 \lambda_n; J_{k_1 - 1}(-\lambda_1^2), J_{k_2}(-\lambda_2^2), \ldots, J_{k_n}(-\lambda_n^2)\Big]: \ran P^\perp \rightarrow \ran P^\perp,
\end{equation}
is a weighted shift matrix (see \eqref{eqn:weighted matrix 2}). Next, since \[
U_{12} = Q_{\ran P^\perp} U|_{\ran P}, \text{ and } U_{21} = Q_{\ran P} U|_{\ran P^\perp},
\]
Lemma \ref{lemma: proj reducing} implies that
\begin{equation}\label{eq: U 1221}
U_{12} U_{21} (P^\perp \mathcal{S}) \subseteq P^\perp \mathcal{S}.
\end{equation}
Now we consider the spectral decomposition of the diagonal operator $U_{11}$ as
\[
\ran P^{\perp} = \bigoplus_{i \in \Lambda}\Big(E_{\sqrt{1-\lambda_i^2}}(U_{11})\Big).
\]
By $P^\perp \cls \neq \{0\}$, we have a non-zero $x \in P^\perp \cls \subseteq \ran P^\perp$. Suppose $x = \oplus_{i \in \Lambda} x_i$, where
\[
x_i \in E_{\sqrt{1-\lambda_i^2}}(U_{11}) \qquad (i \in \Lambda).
\]
Since $\cls$ reduces $(U,P)$ and $U_{11} = P_{\ran P^\perp} U|_{\ran P^\perp}$, it follows from Lemma \ref{lemma: proj reducing} that
\[
U_{11}(P^\perp \cls) = P_{\ran P^\perp} U (P^\perp \cls) \subseteq P_{\ran P^\perp} \cls = P^\perp \cls,
\]
that is, $P^\perp \cls$ is invariant under $U_{11}$. Since $U_{11}$ is an invertible diagonal operator, we are exactly in the setting of Lemma \ref{lemma: diagonal} and hence
\[
x_i \in P^\perp \cls \qquad (i \in \Lambda).
\]
Now since $x \neq 0$, there exists $j \in \Lambda$ such that $x_j \neq 0$.

\bigskip
	
\noindent \textbf{Claim:} Either $f^{j+1}_1$ or $f^1_1$ is in $P^\perp \cls$.
\smallskip

\noindent To prove this claim, let us first represent $x_j$ with respect to $\{f^j_t : t =1, \ldots, k_j\} \in B_{E_{\sqrt{1-\lambda_j^2}}(U_{11})}$ as
\[
x_j = \displaystyle\sum_{t = 1}^{k_j} \alpha_t f^j_t.
\]
Let $t_0$ be the largest value of $t$, $1 \leq t \leq k_j$, such that $\alpha_{t_0} \neq 0$, and let
\[p = k_j - t_0 +1.\] Then
\[x_j = \displaystyle\sum_{t = 1}^{t_0} \alpha_t f^j_t,
\]
and hence
\begin{align*}
(U_{12}U_{21})^{p}(x_j) &= \displaystyle\sum_{t = 1}^{t_0} \alpha_t (U_{12}U_{21})^{p}(f^j_t) \\ &= \displaystyle\sum_{t = 1}^{t_0 - 1} \alpha_t (U_{12}U_{21})^{p}(f^j_t) + \alpha_{t_0} (U_{12}U_{21})^{p}(f^j_{t_0})\\ &= y + z,
\end{align*}
where
\begin{equation}\label{yz}
y = \displaystyle\sum_{t = 1}^{t_0 - 1} \alpha_t (U_{12}U_{21})^{p}(f^j_t), \text{ and } z = \alpha_{t_0} (U_{12}U_{21})^{p}(f^j_{t_0}).
	\end{equation}
Recall from \eqref{eq: U 1221 matrix} that $U_{12} U_{21}$ is a weighted shift matrix. Then, by Lemma \ref{lemma: weighted matrix - diagonal}, we see that
\[
(U_{12} U_{21})^p = \begin{bmatrix}
0 & D_p
\\
D_{m-p} & 0
\end{bmatrix},
\]
where $m = \dim (\ran P^{\perp})$ and $D_p, D_{m-p}$ are respectively $p \times p$ and $(m-p) \times (m-p)$ diagonal matrices with nonzero diagonal entries. In view of the action of $U_{12}U_{21}$ on basis elements as in \eqref{formula U12U21} and keeping in mind that $p = k_j - t_0 +1$, it follows that
\begin{align*}
(U_{12}U_{21})^p(f^j_t) = \begin{cases}
\gamma_t f^j_{t + k_j - t_0 +1} \in E_{\sqrt{1-\lambda_j^2}}(U_{11}) & \mbox{ if } t < t_0
\\
\gamma_{t_0} f^{j+1}_1 \in E_{\sqrt{1-\lambda_{j+1}^2}}(U_{11}) & \mbox{ if } t = t_0 \mbox{ and } j < n
\\
\gamma f^1_1 \in E_{\sqrt{1-\lambda_1^2}}(U_{11}) &  \mbox{ if } t = t_0 \mbox{ and } j = n 		
\end{cases}
\end{align*}
for some non-zero scalars $\gamma_t$, $\gamma_{t_0}$ and $\gamma$. Consequently, it follows from Equation \eqref{yz} that
\[
y \in E_{\sqrt{1-\lambda_j^2}}(U_{11}),
\]
and there exists a non-zero scalar $\tilde{\gamma}$ such that
\[
z = \tilde{\gamma} f^{j+1}_1 \text{ or } \tilde{\gamma} f^1_1,
\]
according as $j < n$ or $j = n$. The above equality ensures that $z \in E_{\sqrt{1-\lambda_{j+1}^2}}(U_{11})$ or $z \in E_{\sqrt{1-\lambda_1^2}}(U_{11})$ according as $j < n$ or $j = n$. In summary
\[
(U_{12}U_{21})^{p}(x_j) = y + z,
\]
with $y \in E_{\sqrt{1-\lambda_j^2}}(U_{11})$, and $z \in E_{\sqrt{1-\lambda_{j+1}^2}}(U_{11})$ or $z \in E_{\sqrt{1-\lambda_1^2}}(U_{11})$ according as $j < n$ or $j = n$. Now since $P^{\perp} \mathcal{S}$ is invariant under $U_{11}$, Lemma 2.6 yields
\[
y \in P^{\perp} \mathcal{S} \mbox{ and } z \in P^{\perp} \mathcal{S}.
\]
Since $z$ is a non-zero scalar multiple of $f^{j+1}_1$ or $f^1_1$ according as $j < n$ or $j = n$, it follows that either $f^{j+1}_1$ or $f^1_1$ is in $P^{\perp} \mathcal{S}$ depending on whether $j < n$ or $j = n$. This completes the proof of the claim.

\smallskip

\noindent Since $U_{12} U_{21}$ is a weighted shift matrix corresponding to $\cup_{i \in \Lambda} \{f^i_t: 1 \leq t \leq k_i\} \in B_{\ran P^\perp}$ (see part (3) of Lemma \ref{lemma: main fd}), it follows from Lemma \ref{shift} that for any $i \in \Lambda$ and any $t$, $1 \leq t \leq k_i$, $f^i_t$ is a cyclic vector for $U_{12} U_{21}$, that is, for any $i \in \Lambda$ and any $t$ with $1 \leq t \leq k_i$,
 \[
 \ran P^\perp = \text{span} \{(U_{12} U_{21})^m f^i_t: m \geq 0\}.
 \]
Since either $f^1_1$ or $f^{j+1}_1$ is in $P^\perp \cls$ we finally obtain that $\ran P^\perp = P^\perp \cls$. The proof for the case $P \cls \neq \{0\}$ is similar.
\end{proof}

\newsection{Diagonals with one positive eigenvalue}\label{section-finite dim 2}

In this section, we focus on all the remaining cases of distinguished diagonal operators on finite-dimensional Hilbert spaces. As we will see, Corollary \ref{equ} and Theorem \ref{thm: main 1}, together with the main result of this section will settle Question \ref{Q1} completely in the case of finite-dimensional Hilbert spaces.

We continue to follow the setting of Lemma \ref{lemma: main fd}.


\begin{theorem}\label{main 2}
Let $\mathcal{E}$ be a finite-dimensional Hilbert space, $T \in \clb(\cle)$ be a distinguished diagonal operator, and assume that $\mbox{dim} E_1(T) = \mbox{dim} E_{-1}(T)$. Then Question \ref{Q1} has an affirmative answer whenever at least one of the following two hypotheses hold:
\begin{enumerate}
  \item $T$ has only one positive eigenvalue in $(0, 1)$.
  \item $1$ is the only positive eigenvalue of $T$ with $\mbox{dim}E_1(T) = 1$.
\end{enumerate}
Moreover, if $1$ is the only positive eigenvalue of $T$ with $\mbox{dim}E_1(T) > 1$, then the answer to Question \ref{Q1} is negative.
\end{theorem}
\begin{proof}
Suppose that $T$ has only one eigenvalue $\lambda$ lying in $(0, 1)$. Then
\[
\sigma(T) = \{\pm \lambda\},
\]
and $\mbox{dim} E_{\lambda}(T) = \mbox{dim} E_{-\lambda}(T)$. Moreover
\[
\cle = E_{-\lambda}(T) \oplus E_{\lambda}(T).
\]
Assume that $\mbox{dim} E_{\lambda}(T) = 1$. Then $\text{dim} \cle = 2$, and hence the distinguished diagonal operator $T \in \clb(\cle)$ has two eigenvalues $\pm \lambda$. Consequently, the simple block constructed in \cite[Example 6.6]{HQY2015} yields an irreducible BCL pair $(M_{\Phi_1}, M_{\Phi_2})$ on $H^2_{\mathbb{C}^2}(\mathbb{D})$ such that the non-zero part of $C(M_{\Phi_1}, M_{\Phi_2})$ is unitarily equivalent to $T$. Let us now assume that
\[
k := \mbox{dim} E_{\lambda}(T) = \mbox{dim} E_{-\lambda}(T) \geq 2.
\]
Let $\{e_t: t = 1, \ldots, k\} \in B_{E_{\lambda}(T)}$, and let $U_\lambda: E_{\lambda}(T) \rightarrow E_{-\lambda}(T)$ is a unitary. Set
\[
f_t := \sqrt{\frac{1+\lambda}{2}} e_t \oplus \sqrt{\frac{1-\lambda}{2}} U_{\lambda} e_t, \text{ and }
\tilde{f}_t := \sqrt{\frac{1-\lambda}{2}} e_t \oplus \Big(- \sqrt{\frac{1+\lambda}{2}}\Big) U_{\lambda}e_t,
\]
for all $t = 1, \ldots, k$. By Lemma \ref{lemma: main fd}, we have
\[
\{\sqrt{1-\lambda^2} f_t \oplus (- \lambda) \tilde{f}_t, \; \lambda f_t \oplus \sqrt{1-\lambda^2} \tilde{f}_t : t =1, \ldots, k\} \in B_{\mathcal{E}},
\]
and $P$ defines a projection on $E_{\lambda}(T) \oplus E_{-\lambda}(T)$, where
\[
P = I_{E_{\lambda}(T) \oplus E_{-\lambda}(T)} - \left[ \begin{array}{cc}
\frac{1+\lambda}{2}I_{E_{\lambda}(T)} & \frac{\sqrt{1-\lambda^2}}{2} U_{\lambda}^*
\\
\frac{\sqrt{1-\lambda^2}}{2} U_{\lambda} & \frac{1-\lambda}{2}I_{E_{-\lambda}(T)}
\end{array}
\right].
\]
Also we know that $\{\tilde{f}_t: t=1, \ldots, k\} \in B_{\ran P}$, and $\{f_t: t=1, \ldots, k\} \in B_{\ran P^{\bot}}$, and
\[
T f_t = \lambda^2 f_t + \lambda \sqrt{1 - \lambda^2} \tilde{f}_t, \text{ and } T \tilde{f}_t = \lambda_i \sqrt{1 - \lambda^2} f_t
- \lambda^2 \tilde{f}_t,
\]
for all $t =1, \ldots, k$. With the projection $P$ as above, we now proceed to construct the required unitary $U: \mathcal{E} \rightarrow \mathcal{E}$. Let $\alpha (\neq 1)$ be a unimodular scalar. Define $U$ on $\ran P^\perp$ by
\[
U {f}_t =
\begin{cases} \alpha((\sqrt{1 - \lambda^2}) f_1 \oplus (- \lambda) \tilde{f}_1) & \mbox{if}~ t =1,
\\
(\sqrt{1 - \lambda^2}) f_t \oplus (- \lambda) \tilde{f}_t & \mbox{if}~ 2 \leq t \leq k,
\end{cases}
\]
and on $\ran P$ by
\[
U \tilde{f}_t =
\begin{cases} \lambda f_{t+1} \oplus (\sqrt{1 - \lambda^2}) \tilde{f}_{t+1} & \mbox{if}~ 1 \leq t < k,
\\
\lambda f_1 \oplus (\sqrt{1 - \lambda^2}) \tilde{f}_1 & \mbox{if}~ t =k.
\end{cases}
\]
As in the proof of Theorem \ref{thm: main 1}, with respect to $\mathcal{E} = \ran\, P^{\bot} \oplus \ran\, P$, set
\begin{equation*}
U = \left[ \begin{array}{cc}
U_{11} & U_{12} \\
U_{21} & U_{22}\\
\end{array}
\right].
\end{equation*}
Then, with respect to $\{\tilde{f}_t: t=1, \ldots, k\} \in B_{\ran P}$ and $\{f_t: t=1, \ldots, k\} \in B_{\ran P^{\bot}}$, it is easy to see that $U_{11} = Q_{\ran P^{\bot}} U|_{\ran P^{\bot}}: \ran P^{\perp} \rightarrow \ran P^{\perp}$ and $U_{21} = Q_{\ran P} U|_{\ran P^{\bot}}: \ran P^{\perp} \rightarrow \ran P$ are diagonal operators with
\[
[U_{11}] = \mbox{diag} \Big(\alpha \sqrt{1-\lambda^2}, \; \Big[\sqrt{1-\lambda^2}\Big]_{k-1}\Big),
\]
and
\[
[U_{21}] = \mbox{diag} \Big(-\alpha \lambda, \; \Big[-\lambda\Big]_{k-1}\Big).
\]
Similarly, $U_{12} = Q_{\ran P^{\perp}} U|_{\ran P}: \ran P \rightarrow \ran P^{\perp}$ and $U_{22} = Q_{\ran P}U|_{\ran P}: \ran P \rightarrow \ran P$ are weighted shift matrices, where
\[
[U_{12}] = \Big[\lambda; J_{k-1}(\lambda)\Big],
\]
and
\[
[U_{22}] = \Big[\sqrt{1-\lambda^2}; J_{k-1}(\sqrt{1-\lambda^2})\Big].
\]
Then the verification of $P^{\perp} - U P^{\perp} U^* = T$, along with the irreducibility of the BCL triple $(\cle, U, P)$, follows exactly the same line of argument as in the proof of Theorem \ref{thm: main 1}. This completes the proof for $\lambda \in (0,1)$ case.

\noindent Now we assume that $\lambda = 1$. We know that $\mathcal{E} = E_{-1}(T) \oplus E_{1}(T)$ and
\begin{equation}\label{eqn: T lambda1}
T = \left[ \begin{array}{cc}
I_{E_{-1}(T)} & 0 \\
0 & -I_{E_{1}(T)}\\
\end{array}
\right].
\end{equation}
For the moment, assume that $T =P^{\perp} - UP^{\perp}U^*$ for some unitary $U$ and projection $P$ on $\mathcal{E}$. Then Theorem \ref{struc} immediately implies that
\begin{equation*}
P^{\perp} = \left[ \begin{array}{cc}
I_{E_{-1}(T)} & 0 \\
0 & 0\\
\end{array}
\right],
 \  \ \mbox{ and } \  \
UP^{\perp}U^* = \left[ \begin{array}{cc}
0 & 0 \\
0 & I_{E_{1}(T)}\\
\end{array}
\right],
\end{equation*}
on $\cle = E_1(T) \oplus E_{-1}(T)$. It is clear from the representations of $P^{\perp}$ and $UP^{\perp}U^*$ that $U E_1(T) = E_{-1}(T)$ and $U E_{-1}(T) = E_1(T)$. Consequently, there exist unitaries $A: E_{1}(T) \rightarrow E_{-1}(T)$ and $B: E_{-1}(T) \rightarrow E_{1}(T)$ such that
\begin{equation*}
U = \left[ \begin{array}{cc}
0 & A \\
B & 0\\
\end{array}
\right].
\end{equation*}
Therefore, given a distinguished diagonal operator $T \in \clb(\cle)$ as in \eqref{eqn: T lambda1}, a BCL triple $(\cle, U, P)$ solves $T =P^{\perp} - UP^{\perp}U^*$ if and only if there exist unitaries $A \in \clb(E_{1}(T), E_{-1}(T))$ and $B \in \clb(E_{-1}(T), E_{1}(T))$ such that
\[
U = \left[ \begin{array}{cc}
0 & A \\
B & 0\\
\end{array}
\right], \text{ and } P = \left[ \begin{array}{cc}
0 & 0 \\
0 & I_{E_{1}(T)}\\
\end{array}
\right].
\]
Thus, we consider a BCL triple $(\cle, U, P)$ with $U$ and $P$ as above. Suppose
\[
1 = \mbox{dim} E_{-1}(T) = \mbox{dim} E_{1}(T).
\]
In particular, $\cle \cong \mathbb{C}^2$. The only non-trivial proper $P$-reducing subspaces of $\cle$ are $E_{1}(T)$ and $E_{-1}(T)$. But neither of these is invariant under $U$, and hence $(\cle, U, P)$ is irreducible. Now we assume that
\[
\mbox{dim} E_{-1}(T) = \mbox{dim} E_{1}(T) \geq 2.
\]
Let $\alpha \in \sigma(U)$, and let $x \in \mathcal{E}$ be an eigenvector corresponding to $\alpha$, that is, $Ux = \alpha x$. Note that $\alpha$ is a unimodular scalar. Write $x= x_1 \oplus x_2 \in E_{-1}(T) \oplus E_{1}(T)$. Then
\[
Ax_2 = \alpha x_1 \quad \mbox{and} \quad Bx_1 = \alpha x_2.
\]
Consider the subspace $\cls = \mbox{span} \{x_1, x_2\}$. Since $A$ and $B$ are unitaries, it follows that $\cls$ reduces both $U$ and $P$. Finally, $\text{dim} \cle \geq 4$ implies that $\cls$ is a proper non-trivial subspace of $\cle$ and this completes the proof of the theorem.
\end{proof}

We summarize all the results obtained so far for finite-dimensional Hilbert spaces (Corollary \ref{equ}, Theorem \ref{thm: main 1}, and Theorem \ref{main 2}) in the following theorem.

\begin{theorem}\label{main}
Let $\mathcal{E}$ be a finite-dimensional Hilbert space and let $T \in \clb(\cle)$ be a distinguished diagonal operator. If
\[
\mbox{dim}\, E_1(T) \neq \mbox{dim}\, E_{-1}(T),
\]
then it is not possible to find a BCL pair $(M_{\Phi_1}, M_{\Phi_2})$ on $H^2_{\cle}(\mathbb{D})$ such that the non-zero part of $C(M_{\Phi_1}, M_{\Phi_2})$ is unitarily equivalent to $T$. If
\[
\mbox{dim} E_1(T) = \mbox{dim} E_{-1}(T),
\]
then there exists an irreducible BCL pair $(M_{\Phi_1}, M_{\Phi_2})$ on $H^2_{\cle}(\mathbb{D})$ such that $C(M_{\Phi_1}, M_{\Phi_2})|_{\cle} = T$ whenever at least one of the following three hypotheses hold:
\begin{enumerate}
\item $T$ has at least two distinct positive eigenvalues.
  \item $T$ has only one positive eigenvalue in $(0, 1)$.
  \item $1$ is the only positive eigenvalue of $T$ with $\mbox{dim}E_1(T) = 1$.
\end{enumerate}
Moreover, if $1$ is the only positive eigenvalue of $T$ and
\[
\mbox{dim} E_1(T) > 1,
\]
then it is not possible to find any irreducible BCL pair $(M_{\Phi_1}, M_{\Phi_2})$ on $H^2_{\cle}(\mathbb{D})$ such that the non-zero part of $C(M_{\Phi_1}, M_{\Phi_2})$ is unitarily equivalent to $T$.
\end{theorem}

Therefore, Theorem \ref{main} settles Question \ref{Q1} completely in the finite-dimensional case.

\section{Diagonals on infinite-dimensional spaces}\label{section-infinite dim}

In this section, we analyse Question \ref{Q1} for distinguished diagonal operators acting on infinite-dimensional Hilbert spaces. As we have seen in Corollary \ref{equ} (or see Theorem \ref{main}), for a finite-dimensional Hilbert space $\mathcal{E}$ and a distinguished diagonal operator $T \in \clb(\cle)$, if
\[
\mbox{dim}\, E_1(T) - \mbox{dim}\, E_{-1}(T) \neq 0,
\]
then Question \ref{Q1} is always negative. Here, however, at the other extreme, we prove that if $\cle$ is an infinite-dimensional space, then under the above assumption Question \ref{Q1} still could be affirmative. For instance, we prove that if $T$ is a distinguished diagonal operator acting on an infinite-dimensional Hilbert space $\cle$, and if
\[
|\mbox{dim} E_1(T) - \mbox{dim} E_{-1}(T)| \leq 1,
\]
then Question \ref{Q1} is always in the affirmative.

We begin by showing that Question \ref{Q1} is in the affirmative whenever $\mbox{dim} E_1(T) = \mbox{dim} E_{-1}(T)$ (\mbox{may be zero also}). Part of the proof proceeds along the lines of the proof of Theorem \ref{thm: main 1}. Those similarities will be pointed out and subsequently omitted in what follows.

\begin{theorem}\label{inf}
Let $\cle$ be an infinite-dimensional Hilbert space, and let $T \in \clb(\cle)$ be a distinguished diagonal operator. If
\[
\mbox{dim} E_1(T) = \mbox{dim} E_{-1}(T),
\]
then there exists an irreducible BCL pair $(M_{\Phi_1}, M_{\Phi_2})$ on ${H}^2_{\mathcal{E}}(\D)$ such that $C(M_{\Phi_1}, M_{\Phi_2})|_{\cle} = T$.
\end{theorem}
\begin{proof}
Let $\sigma(T) \cap (0,1] = \{\lambda_n : n \in \mathbb{Z}\}$. Since $\mbox{dim} E_1(T) = \mbox{dim} E_{-1}(T)$, it follows that $\sigma(T) = \{ \pm \lambda_n : n \in \mathbb{Z}\}$ (see \eqref{eqn - plus minus lambda}), and
\[
k_n:= \mbox{dim} E_{\lambda_n}(T) = \mbox{dim} E_{-\lambda_n}(T) < \infty \qquad (n \in \mathbb{Z}).
\]
It also follows that
\[
\mathcal{E} = \bigoplus_{n \in \mathbb{Z}}\Big(E_{\lambda_n}(T) \oplus E_{-\lambda_n}(T)\Big).
\]
Fix $n \in \mathbb{N}$ and $\{e^n_t: 1 \leq t \leq k_n \} \in B_{E_{\lambda_n}(T)}$. Then, as in Lemma \ref{lemma: main fd}, there exists a unitary $U_n: E_{\lambda_n}(T) \rightarrow E_{-\lambda_n}(T)$ such that $\{U_n e^n_t: 1 \leq t \leq k_n \} \in B_{E_{-\lambda_n}(T)}$, and
\begin{equation*}
Q_n = \left[ \begin{array}{cc}
\frac{1+\lambda_n}{2}I_{E_{\lambda_n}(T)} & \frac{\sqrt{1-\lambda_n^2}}{2} U_n^*
\\
\frac{\sqrt{1-\lambda_n^2}}{2} U_n & \frac{1-\lambda_n}{2}I_{E_{-\lambda_n}(T)}
\end{array}
\right],
\end{equation*}
defines a projection on $E_{\lambda_n}(T) \oplus E_{-\lambda_n}(T)$. Moreover, $\{f^n_t : 1 \leq t \leq k_n \} \in B_{\ran\, Q_n}$ and $\{\tilde{f}^n_t : 1 \leq t \leq k_n \} \in B_{\ran\, Q_n^{\perp}}$, where
\[
f^n_t := \sqrt{\frac{1+\lambda_n}{2}} e^n_t \oplus \sqrt{\frac{1-\lambda_n}{2}} U_n e^n_t, \text{ and }
\tilde{f}^n_t := \sqrt{\frac{1-\lambda_n}{2}} e^n_t \oplus \Big(-\sqrt{\frac{1+\lambda_n}{2}}\Big) U_n e^n_t,
\]
for all $t=1, \ldots, k_n$. Consider the projection $P := (\oplus_{n \geq 1} Q_n)^\perp \in \clb(\mathcal{E})$. It follows that
\begin{equation}\label{eqn: two unions bases}
\cup_{n \in \mathbb{Z}} \{f^n_t: 1 \leq t \leq k_n \} \in B_{\ran P^{\perp}}, \text{ and } \cup_{n \in \mathbb{Z}} \{\tilde{f}^n_t: 1 \leq t \leq k_n \} \in B_{\ran P}.
\end{equation}
Define the unitary $U: \mathcal{E} \rightarrow \mathcal{E}$ by specifying
\[
U(f^n_t) = \sqrt{1 - \lambda_n^2} f^n_t \oplus (- \lambda_n) \tilde{f}^n_t,
\]
and
\[
U(\tilde{f}^n_t) =
\begin{cases} \lambda_n f^n_{t+1} \oplus \sqrt{1 - \lambda_n^2} \tilde{f}^n_{t+1} & \mbox{if}~ 1 \leq t < k_n
\\
\lambda_{n+1} f^{n+1}_1 \oplus \sqrt{1 - \lambda_{n+1}^2} \tilde{f}^{n+1}_1 & \mbox{if}~ t = k_n,
\end{cases}
\]
for all $1 \leq t \leq k_n$ and $n\geq 1$. It is easy to see that $P^{\perp} - U P^{\perp} U^* = T$. Suppose
\[
U = \left[ \begin{array}{cc}
U_{11} & U_{12} \\
U_{21} & U_{22}\\
\end{array}
\right],
\]
with respect to $\mathcal{E} = \ran P^{\bot} \oplus \ran P$. It is clear from the definition of $U$ that with respect to the orthonormal bases of $\ran P^{\perp}$ and $\ran P$ as in \eqref{eqn: two unions bases}, the matrices of  $U_{11} : \ran P^{\perp} \to \ran P^{\perp}$ and $U_{21} : \ran P^{\perp} \to \ran P$ are diagonal:
\begin{equation}\label{U11U21}
[U_{11}] = \mbox{diag} \big([\sqrt{1-\lambda_n^2}]_{k_n} \big)_{n \in \mathbb{Z}}, \text{ and } [U_{21}] = \mbox{diag} \big([-\lambda_n]_{k_n} \big)_{n \in \mathbb{Z}},
\end{equation}
and $U_{12} U_{21} : \ran P^{\perp} \to \ran P^{\perp}$ is a weighted shift defined by
\begin{align*}
U_{12}U_{21}(f^n_t) = \begin{cases}
-\lambda_n^2 f^n_{t+1} \ &\mbox{if} \ 1 \leq t < k_n\\
-\lambda_n \lambda_{n+1} f^{n+1}_1 \  &\mbox{if} \ t = k_n.
\end{cases}
\end{align*}
Now let $\mathcal{S}$ be a non-zero closes subspace of $\cle$. Assume that $\cls$ reduces $(U, P)$. In particular, $\cls$ reduces $P$, and hence, we may write
\[
\cls = P^{\perp}\mathcal{S} \oplus P\mathcal{S}.
\]
Assume, without loss of generality, that $P^{\perp}\mathcal{S} \neq \{0\}$. It is enough to prove that $\cls = \cle$ (as the $P \mathcal{S} \neq \{0\}$ case would follow similarly). However:

\smallskip

\noindent \textsf{Claim:} If $P^{\perp}\mathcal{S} = \ran P^{\perp}$, then $\cls = \cle$.

\noindent To prove the claim we assume that $P^{\perp}\mathcal{S} = \ran P^{\perp}$. Then $U_{21}(P^{\perp}\mathcal{S}) \subseteq P \mathcal{S}$ (see \eqref{eq: U21 P perp S} in the proof of Theorem \ref{thm: main 1}) and the matrix representation of $U_{21}$ imply that
\[
\cup_{n \in \mathbb{Z}}\{\tilde{f}^n_t : 1 \leq t \leq k_n\} \subseteq P \mathcal{S}.
\]
On the other hand, since $\cup_{n \in \mathbb{Z}}\{\tilde{f}^n_t : 1 \leq t \leq k_n\} \in B_{\ran P}$ and $P \mathcal{S}$ is a closed subspace of $\ran P$, it follows that $\ran P = P \mathcal{S}$. Then $\cls = \ran P^\perp \oplus \ran P = \cle$, from which the claim follows immediately.

\smallskip

\noindent Therefore, all we need to check is the fact that
\[
P^{\perp}\mathcal{S} = \ran P^{\perp}.
\]
Again, as in the proof of Theorem \ref{thm: main 1} (see \eqref{eq: U 1221 matrix} and \eqref{eq: U 1221}), since $P^{\perp}\mathcal{S}$ reduces $U_{12} U_{21}$, and $U_{12} U_{21}$ is a weighted shift on $\ran P^{\perp}$, Lemma \ref{shift} would prove the above equality if we can show that $f^{n}_1 \in P^{\perp}\mathcal{S}$ for some $n$. To this end, consider a non-zero vector $x \in P^{\perp}\mathcal{S}$. As $U_{11} \in \clb(\ran P^{\perp})$ is diagonalizable with eigenvalues $\{\sqrt{1-\lambda_{n}^2} : n \in \mathbb{Z}\}$, it follows that
\[
\ran P^{\perp} = \oplus_{n \in \mathbb{Z}} E_{\sqrt{1-\lambda_n^2}}(U_{11}),
\]
and hence $x = \sum_{n \in \mathbb{Z}} x_n$, where
\[
x_n \in E_{\sqrt{1-\lambda_n^2}}(U_{11}).
\]
Since $P^{\perp}\mathcal{S}$ reduces $U_{11}$, Lemma \ref{lemma: diagonal} yields that $x_n \in P^{\perp}\mathcal{S}$ for all $n \in \mathbb{Z}$. Choose $m$ such that $x_m \neq 0$ and let
\[
x_m = \sum_{t = 1}^{k_m} \alpha_t f^m_t.
\]
If $t_0 = \max \{t: \alpha_{t} \neq 0, 1 \leq t \leq k_m\}$, then a similar argument as in the proof of Theorem \ref{thm: main 1} shows that
\[
(U_{12} U_{21})^{k_m - t_0 +1} f^m_s \in E_{\sqrt{1-\lambda_m^2}}(U_{11}),
\]
for all $s < t_0$, and, there exists a non-zero scalar $\alpha$ such that
\[
(U_{12} U_{21})^{k_m - t_0 +1}(f^m_{t_0}) = \alpha f^{m+1}_1.
\]
Then we conclude, proceeding again along the same line of argument as in the proof of Theorem \ref{thm: main 1}, that $f^{m+1}_1 \in P^{\perp}\mathcal{S}$. Since the proof of the other case that $P \cls \neq \{0\}$ is also similar, this completes the proof.
\end{proof}

Now we prove that  Question \ref{Q1} is in the affirmative whenever $|\mbox{dim} E_1(T) - \mbox{dim} E_{-1}(T)| = 1$. However, unlike the above theorem (and except for the general idea), the proof of the present case is different from that of Theorem \ref{thm: main 1}. In other words, the irreducible BCL triple constructed below is fairly different from those constructed in Theorem \ref{main} and Theorem \ref{inf} above and requires more effort.

\begin{theorem}\label{diff1}
Let $\cle$ be an infinite-dimensional Hilbert space, and let $T \in \clb(\cle)$ be a distinguished diagonal operator. If
\[
\mbox{dim}\, E_1(T) = \mbox{dim}\, E_{-1}(T) \pm 1,
\]
then there exists an irreducible BCL pair $(M_{\Phi_1}, M_{\Phi_2})$ on ${H}^2_{\mathcal{E}}(\D)$ such that $C(M_{\Phi_1}, M_{\Phi_2})|_{\cle} = T$.
\end{theorem}
\begin{proof}
Assume, without loss of generality, that $\mbox{dim}\, E_{-1}(T) = \mbox{dim}\, E_1(T) + 1$. Further, assume that $\mbox{dim}\, E_1(T) > 0$, that is, $\lambda_0:= 1 \in \sigma(T)$, and set $\sigma(T) \cap (0, 1) = \{\lambda_n : n \geq 1 \}$. Then $\sigma(T) = \{\pm \lambda_n : n \geq 0 \}$. Also, set $k_0=\mbox{dim}\, E_1(T)$ so that	
\[
\mbox{dim}\, E_{-1}(T) = k_0 + 1,
\]
and let $\{f^0_t : 1 \leq t \leq k_0 \} \in B_{E_1(T)}$ and $\{\tilde{f}^0_t : 1 \leq t \leq k_0 + 1 \} \in B_{E_{-1}(T)}$. For each $n \geq 1$, we use the same notations used in the proof of Theorem \ref{inf}: $k_n := \mbox{dim}\, E_{\lambda_n}(T)$, $U_n: E_{\lambda_n}(T) \rightarrow E_{-\lambda_n}(T)$ is a unitary, $\{e^n_t : 1 \leq t \leq k_n \} \in B_{E_{\lambda_n}(T)}$, and
\[
f^n_t = \sqrt{\frac{1 + \lambda_n}{2}} e^n_t \oplus \sqrt{\frac{1 - \lambda_n}{2}} U_n e^n_t, \text{ and } \tilde{f}^n_t = \sqrt{\frac{1 - \lambda_n}{2}} e^n_t \oplus \Big(- \sqrt{\frac{1 + \lambda_n}{2}} U_n e^n_t\Big)
\]
for all $1 \leq t \leq k_n$. For notational convenience, we let
\[
\clf = \displaystyle\cup_{m \geq 0} \{f^m_t : 1 \leq t \leq k_m\}, \text{ and } \tilde\clf = \cup_{n \geq 1} \{\tilde{f}^n_t : 1 \leq t \leq k_n \Big \} \cup \{ \tilde{f}^0_t : 1 \leq t \leq k_0 + 1\}.
\]
Note that our goal is to construct an irreducible BCL triple $(\cle, U, P)$ such that $P^{\perp} - U P^{\perp} U^* = T$. Clearly
\[
\clf \cup \tilde \clf \in B_{\cle}.
\]
We simply consider the projection $P \in \clb(\cle)$ such that $\clf \in B_{\ran P^{\perp}}$ and $\tilde \clf \in B_{\ran P}$. The construction of $U$ on $\cle$, however, needs more care. We proceed as follows: On $\clf$, define $U$ as
\[
U f^n_t =
\begin{cases}
\tilde{f}^0_{t+1} & \text{if } n=0 \text{ and } 1 \leq t \leq k_0
\\
\sqrt{1 - \lambda_n^2}f^n_{t+1} \oplus (- \lambda_n) \tilde{f}^n_{t+1} & \text{if } n \geq 1 \text{ and } 1 \leq t < k_n
\\
\sqrt{1 - \lambda_{n-1}^2}f^{n-1}_1 \oplus (-\lambda_{n-1}) \tilde{f}^{n-1}_1 & \text{if } n \geq 1 \text{ and } t = k_n,
\end{cases}
\]
and on $\tilde\clf$, we define
\[
U \tilde{f}^n_t =
\begin{cases}
f^0_t & \text{if } n=0 \text{ and } 1 \leq t \leq k_0
\\
\lambda_1 f^1_1 + \sqrt{1 - \lambda_1^2}\tilde{f}^1_1 & \text{if } n=0 \text{ and } t = k_0+1
\\
\lambda_n f^n_{t+1} \oplus \sqrt{1 - \lambda_n^2}  \tilde{f}^n_{t+1} & \text{if } n \geq 1 \text{ and } 1 \leq t < k_n
\\
\lambda_{n+1}f^{n+1}_1 \oplus \sqrt{1 - \lambda_{n+1}^2} \tilde{f}^{n+1}_1 & \text{if } n \geq 1 \text{ and } t = k_n.
\end{cases}
\]
It is now clear from the definition of $U$ and $P$ that $P^{\perp} - U P^{\perp} U^* = T$. Suppose
\begin{equation*}
U = \left[ \begin{array}{cc}
U_{11} & U_{12} \\
U_{21} & U_{22}\\
\end{array}
\right],
\end{equation*}
on $\mathcal{E} = \ran P^{\perp} \oplus \ran P$. Since $U_{11} = Q_{\ran P^{\perp}} U|_{\ran P^{\perp}}$, we have
\[
U_{11} f^n_t =
\begin{cases}
0 & \text{if } n=0 \text{ and } 1 \leq t \leq k_0
\\
\sqrt{1 - \lambda_n^2} f^n_{t+1} & \text{if } n \geq 1 \text{ and } 1 \leq t < k_n
\\
\sqrt{1 - \lambda_{n-1}^2} f^{n-1}_1 & \text{if } n \geq 1 \text{ and } t = k_n,
\end{cases}
\]
and hence
\[
U_{11}^* U_{11} f^n_t =
\begin{cases}
0 & \text{if } n=0 \text{ and } 1 \leq t \leq k_0
\\
(1 - \lambda_n^2) f^n_{t} & \text{if } n \geq 1 \text{ and } 1 \leq t < k_n
\\
(1 - \lambda_{n-1}^2) f^{n}_{k_n} & \text{if } n \geq 1 \text{ and } t = k_n.
\end{cases}
\]
In particular, $U_{11}^* U_{11}$ is a diagonalizable operator with $\sigma(U_{11}^* U_{11}) = \{1 - \lambda_n^2\}_{n \geq 0}$. Therefore 
\[
\{f^0_t : 1 \leq t \leq k_0 \} \cup \{f^1_{k_1} \} \in B_{E_0(U_{11}^* U_{11})} = B_{E_{1 - \lambda_0^2}(U_{11}^* U_{11})},
\]
and, for all $n \geq 1$, we have
\begin{equation}\label{basis of non-zero}
\{f^n_t : 1 \leq t < k_n \} \cup  \{f^{n+1}_{k_{n+1}} \} \in B_{E_{1 - \lambda_n^2}(U_{11}^* U_{11})}.
\end{equation}
Now we prove that $(U,P)$ is irreducible. Suppose $\mathcal{S} \subseteq \cle$ is a non-zero closed subspace, and suppose that $\cls$ reduces $(U, P)$. Then, as before, we write
\[
\mathcal{S} = P^{\perp}\mathcal{S} \oplus P\mathcal{S}.
\]
Assume, without loss of generality, that $P^{\perp}\mathcal{S} \neq \{0\}$ (as the other case that $P\mathcal{S} \neq \{0\}$ would follow similarly). Our goal is to show that $\mathcal{S} = \mathcal{E}$.

\smallskip

\noindent\textbf{Claim:} $f^n_t \in P^{\perp} \mathcal{S}$ for some $n \geq 1$ and $1 \leq t \leq k_n$.

\smallskip
	
\noindent\textit{Proof of the claim:} Pick a non-zero $x \in P^{\perp}\mathcal{S}$, and suppose $x = \mathop\oplus_{n \geq 0}x_n$, where $x_n \in E_{1 - \lambda_n^2}(U_{11}^* U_{11})$ for all $n \geq 0$. Since $P^{\perp}\mathcal{S}$ reduces $U_{11}^* U_{11}$, Lemma \ref{lemma: diagonal} implies (as in the proof of Theorem \ref{thm: main 1}) that $x_n \in P^{\perp}\mathcal{S}$, $n \geq 0$. Let $n_0$ be the smallest non-negative integer such that $x_{n_0} \neq 0$.

\noindent \textit{Case $1$.} Suppose $n_0 \geq 1$. Using the above orthonormal basis of $E_{1 - \lambda_{n_0}^2}(U_{11}^* U_{11})$, represent $x_{n_0}$ as
\[
x_{n_0} = \sum_{t = 1}^{k_{n_0} -1} \alpha^{n_0}_t f^{n_0}_t + \beta f^{n_0 +1}_{k_{n_0 + 1}},
\]
for some scalars $\alpha^{n_0}_t$ and $\beta$. If $\alpha^{n_0}_t = 0$ for all $t$ and $1 \leq t < k_{n_0}$, then clearly $\beta \neq 0$ and hence, $f^{n_0 +1}_{k_{n_0} + 1} \in \mathcal{S}_1$. Suppose $\alpha^{n_0}_t$ are not all zero. Let $t_0$ be the maximum value of $t, 1 \leq t \leq k_{n_0}-1$, such that $\alpha^{n_0}_t \neq 0$. Then
\[
x_{n_0} = \sum_{t = 1}^{t_0} \alpha^{n_0}_t f^{n_0}_t + \beta f^{n_0 +1}_{k_{n_0 + 1}}.
\]
Since $U_{11} (P^{\perp} \cls) \subseteq P^{\perp} \cls$, it follows that $U_{11}^{k_{n_0} - t_0}(x_{n_0}) \in P^{\perp}\mathcal{S}$. The action of $U_{11}$ on $\clf$ now yields
\[
U_{11}^{k_{n_0} - t_0} f^{m}_t =
\begin{cases}
(\sqrt{1 - \lambda_{n_0}^2})^{k_{n_0} - t_0} f^{n_0}_{k_{n_0}} & \text{if } m = n_0 \text{ and } t = t_0
\\
\gamma_t f^{n_0}_{k_{n_0} - t_0 + t} & \text{if } m = n_0 \text{ and } 1 \leq t < t_0
\\
\gamma_{k_{n_0+1}} f^{n_0}_{k_{n_0} - t_0} & \text{if } m = n_0+1 \text{ and } t = k_{n_0+1},
\end{cases}
\]
for some scalars $\gamma_t$ and $\gamma_{k_{n_0+1}}$. In particular, $U_{11}^{k_{n_0} - t_0} f^{n_0}_{t_0} = (\sqrt{1 - \lambda_{n_0}^2})^{k_{n_0} - t_0} f^{n_0}_{k_{n_0}}$, and
\[
U_{11}^{k_{n_0} - t_0} f^{n_0}_t, U_{11}^{k_{n_0} - t_0} f^{n_0 + 1}_{k_{n_0+1}} \in \mbox{span} \{f^{n_0}_1, \cdots, f^{n_0}_{k_{n_0} - 1}\},
\]
for $1 \leq t < t_0$. Then
\[
\begin{split}
U_{11}^{k_{n_0} - t_0} x_{n_0} & = \sum_{t = 1}^{t_0} \alpha^{n_0}_t U_{11}^{k_{n_0} - t_0} f^{n_0}_t + \beta U_{11}^{k_{n_0} - t_0} f^{n_0 +1}_{k_{n_0 + 1}}
\\
& = (\sum_{t = 1}^{t_0-1} \alpha^{n_0}_t U_{11}^{k_{n_0} - t_0} f^{n_0}_t + \beta U_{11}^{k_{n_0} - t_0} f^{n_0 +1}_{k_{n_0 + 1}}) + \alpha^{n_0}_{t_0} (\sqrt{1 - \lambda_{n_0}^2})^{k_{n_0} - t_0}f^{n_0}_{k_{n_0}} ,
\end{split}
\]
and, by \eqref{basis of non-zero}, it follows that
\[
U_{11}^{k_{n_0} - t_0} x_{n_0} \in \mbox{span} \{f^{n_0}_1, \cdots, f^{n_0}_{k_{n_0} - 1}\} \oplus \mbox{span} \{f^{n_0}_{k_{n_0}}\} \subseteq E_{1 - \lambda_{n_0}^2}(U_{11}^* U_{11}) \oplus E_{1 - \lambda_{n_0 -1}^2}(U_{11}^* U_{11}).
\]
As $\alpha^{n_0}_{t_0} (\sqrt{1 - \lambda_{n_0}^2})^{k_{n_0} - t_0}\neq 0$, this implies by an appeal to Lemma \ref{lemma: diagonal} that  $f^{n_0}_{k_{n_0}} \in P^{\perp}\mathcal{S}$ and proves the claim.
	
\smallskip

\noindent \textit{Case $2$.} Suppose $n_0 = 0$. Then $x_0 \in E_0(U_{11}U_{11}^*)$, and hence (see the basis preceding \eqref{basis of non-zero})
\[
x_0 = \sum_{t = 1}^{k_0} \alpha^0_t f^0_t + \beta f^1_{k_1},
\]
for some scalars $\beta$ and $\alpha^0_t$. By the definition of $U_{11}$, we have
\[
U_{11}^* x_0 = \begin{cases}
\beta \sqrt{1 - \lambda_1^2} f^2_{k_2} & \mbox{if }k_1 = 1
\\
\beta \sqrt{1 - \lambda_1^2} f^1_{k_1 -1} & \mbox{if } k_1 > 1.
\end{cases}
\]
Therefore, if $\beta \neq 0$, then $U_{11}^*(P^{\perp}\mathcal{S}) \subseteq P^{\perp}\mathcal{S}$ yields that $f^1_{k_1 -1}$ or $f^2_{k_2}$ is in $P^{\perp}\mathcal{S}$ according as $k_1 > 1$ or $k_1 = 1$. Suppose now that $\beta = 0$ and let $t_0 = \max\{t: \alpha^0_t \neq 0 \}$. Then $x_0 = \sum_{t = 1}^{t_0} \alpha^0_t f^0_t$. By the definition of $U$ on $\clf$, it follows that
\[
U^{2(k_0 - t_0+1)}f^0_t = f^0_{t + k_0 - t_0 + 1} \in \mbox{span}\{f^0_1, \cdots, f^0_{k_0}\},
\]
for all $ 1 \leq t < t_0$, and
\[
U^{2(k_0 - t_0+1)} f^0_{t_0} = \lambda_1 f^1_1 + \sqrt{1 - \lambda_1^2} \tilde{f}^1_1.
\]
Consequently, $U^{2(k_0 - t_0+1)} x_0 \in \cls$ as
\begin{align*}
U^{2(k_0 - t_0+1)} x_0 & = \sum_{t = 1}^{t_0 - 1} \alpha_t^0 f^0_{t + k_0 - t_0 + 1} + \alpha^0_{t_0}(\lambda_1 f^1_1 + \sqrt{1 - \lambda_1^2} \tilde{f}^1_1) \\
& = (\sum_{t = 1}^{t_0 - 1} \alpha_t^0 f^0_{t + k_0 - t_0 + 1} + \alpha^0_{t_0}\lambda_1 f^1_1) + \alpha^0_{t_0}\sqrt{1 - \lambda_1^2} \tilde{f}^1_1.
\end{align*}
As $\mathcal{S}$ is invariant under $P$, it follows that
\[
P U^{2(k_0 - t_0+1)}(x_0) = \alpha^0_{t_0}\sqrt{1 - \lambda_1^2} \tilde{f}^1_1 \in P\mathcal{S},
\]
that is, $\tilde{f}^1_1 \in P \mathcal{S}$. By the definition of $U$ on $\clf$, we have in  particular that
\[
U \tilde{f}^1_1 = \begin{cases}
\lambda_2 f^2_1 + \sqrt{1 - \lambda_2^2}\tilde{f}^2_1 & \mbox{if }k_1 = 1
\\
\lambda_1 f^1_2 + \sqrt{1 - \lambda_1^2} \tilde{f}^1_2 & \mbox{if } k_1 > 1.
\end{cases}
\]
Therefore, we have that either $f^1_2$ or $f^2_1$ in $P^{\perp}\mathcal{S}$. We conclude that, in either case, $f^n_t \in P^{\perp}\mathcal{S}$ for some $n \geq 1$ and $1 \leq t \leq k_n$. This completes the proof of the claim.

\smallskip
	
\noindent Therefore, we can fix $f^m_t \in P^{\perp}\mathcal{S}$ for some $1 \leq t \leq k_m$ and $m \geq 1$. Since $\clf \in B_{\ran P^\perp}$, the definition of $U$ on $\clf$ implies that there exists a non-zero scalar $c$ such that
\begin{align*}
\Big(U({P^{\perp}}U)^{\sum_{i=1}^{m} k_i -t}\Big) f^m_t = c \tilde{f}^0_1,
\end{align*}
and hence, $\tilde{f}^0_1 \in \mathcal{S}$. Since $\mathcal{S}$ is invariant under $U$, applying $U$ repeatedly on $\tilde{f}^0_1$ we see that
\[
\{f^0_t : 1 \leq t \leq k_0\}  \cup  \{\tilde{f}^0_t : 1 \leq t \leq k_0 +1\} \subseteq \mathcal{S}.
\]
Similarly, since $\tilde \clf \in B_{\ran P}$, by a repeated application of the definition of $U$ on $\tilde \clf$ implies
\[
(PU)^t \tilde{f}^0_{k_0 +1} = \mbox{a non-zero scalar multiple of } \tilde{f}^1_t,
\]
for all $1 \leq t \leq k_1$, and
\[
\Big((PU)^{\sum_{i=1}^{n-1} k_i + t}\Big) \tilde{f}^0_{k_0 +1} = \mbox{a non-zero scalar multiple of} \ \tilde{f}^n_t,
\]
for all $1 \leq t \leq k_n$ and $n \geq 1$. Combining the last three observations, we deduce that
\[
\{f^0_t : 1 \leq t \leq k_0\}  \cup  \tilde \clf \subseteq \mathcal{S}.
\]
At this point, we note that it is enough to prove that
\[
\cup_{n \in \mathbb{N}} \{f^n_t : 1 \leq t \leq k_n\} \subseteq \mathcal{S}.
\]
as that would imply that $\mathcal{S}$ contains the orthonormal basis $\clf \cup \tilde \clf \in B_{\cle}$ and completes the proof of the fact that $\cls = \cle$. To this end, again using the definition of $U$ on $\tilde \clf$, for each $1 \leq t \leq k_n$ and $n \geq 1$, we find
\[
f^n_t = \begin{cases}
\frac{1}{\lambda_1} (P^{\perp}U)\tilde{f}^0_{k_0 + 1} & \text{ if } n = t = 1
\\
\frac{1}{\lambda_n} (P^{\perp}U)\tilde{f}^n_t & \text{ if }  1 < t \leq k_n
\\
\frac{1}{\lambda_n} (P^{\perp}U)\tilde{f}^{n-1}_{k_{n-1}} & \text{ if }  t = 1 \text{ and }  n > 1.
\end{cases}		
\]
Since $\mathcal{S}$ reduces $(U, P)$, we finally conclude that $\cup_{n \in \mathbb{N}} \{f^n_t : 1 \leq t \leq k_n\} \subseteq \mathcal{S}$. The proof of the case when $1$ is not an eigenvalue of $T$ (that is, $k_0 = 0$ case) works exactly along the same lines.
\end{proof}

\section{Concluding remarks}\label{section-conclusion}

In summary, the main results of this paper gives a complete answer (sometimes in the affirmative and sometimes in the negative) to Question \ref{Q1} except for the case of infinite-dimensional Hilbert spaces $\cle$ for which
\[
| \mbox{dim} E_1(T) - \mbox{dim} E_{-1}(T) | \geq 2.
\]
In addition, Theorem \ref{diff1} points out a crucial difference between the finite and infinite-dimensional cases: If $T \in \clb(\cle)$ is a distinguished diagonal operator, then the equality $\mbox{dim}\, E_1(T) = \mbox{dim}\, E_{-1}(T)$ is a necessary condition for the existence of an irreducible BCL pair $(V_1, V_2)$ on $H^2_{\cle}(\mathbb{D})$ such that $C(V_1, V_2)|_{\cle} = T$, only when $\mathcal{E}$ is finite-dimensional.

Now we return to the original question of He, Qin, and Yang \cite[page 18]{HQY2015}. As pointed out in the paragraph following Question \ref{Q1}, all the affirmative answers in this paper also yield an affirmative answers to the question of He, Qin, and Yang. More specifically, suppose $T \in \clb(\cle)$ is a distinguished diagonal operator. If $\cle$ is finite-dimensional and
\[
\mbox{dim} E_1(T) = \mbox{dim} E_{-1}(T),
\]
then there exists an irreducible BCL pair $(M_{\Phi_1}, M_{\Phi_2})$ on $H^2_{\cle}(\mathbb{D})$ such that
\[
C(M_{\Phi_1}, M_{\Phi_2})|_{(\ker C(M_{\Phi_1}, M_{\Phi_2}))^\perp} = T,
\]
whenever at least one of the following three hypotheses hold:
\begin{enumerate}
\item $T$ has at least two distinct positive eigenvalues.
\item $T$ has only one positive eigenvalue in $(0, 1)$.
\item $1$ is the only positive eigenvalue of $T$ with $\mbox{dim}E_1(T)= 1$.
\end{enumerate}
If $\cle$ is infinite-dimensional, then the same conclusion holds whenever
\[
| \mbox{dim} E_1(T) - \mbox{dim} E_{-1}(T) | \leq 1.
\]

Finally, we remark that the general questions considered in this paper are those which are fairly routine in the theory of single isometries but appear to be somewhat challenging in the theory of pairs of commuting isometries. Moreover, the complication involved in the range of our answers seems to further indicate the intricate structure of pairs of commuting isometries and shift invariant subspaces of the Hardy space over the bidisc \cite{Yang-S}.

\vspace{0.2in}

\noindent\textbf{Acknowledgement:} We are grateful to the referee for a careful reading of the manuscript. The first author is grateful to the Indian Statistical Institute Bangalore for offering a visiting position. The research of the first named author is supported in part by NBHM (National Board of Higher Mathematics, India) Post Doctoral fellowship File No. 0204/21/2018/R\&D-II/3030. The second named author is supported in part by the Mathematical Research Impact Centric Support, MATRICS (MTR/2017/000522), and Core Research Grant (CRG/2019/000908), by SERB (DST), Government of India. The research of the third named author is supported by NBHM (National Board of Higher Mathematics, India) Post Doctoral fellowship File No. 0204/26/2019/R\&D-II/12037.

\end{document}